\documentclass[11pt]{article}
\usepackage{amsmath}
\usepackage{amssymb,amsthm,latexsym}
\usepackage{xypic,graphicx}

\theoremstyle{plain}
 \newtheorem{theorem}{Theorem}[section]
 \newtheorem{lemma}[theorem]{Lemma}
 \newtheorem{proposition}[theorem]{Proposition}
 \newtheorem{corollary}[theorem]{Corollary}

 \theoremstyle{definition}
 \newtheorem{definition}[theorem]{Definition}

 \theoremstyle{remark}
 
 \newtheorem{remark}[theorem]{Remark}

\DeclareMathOperator{\loc}{loc}
\DeclareMathOperator{\erfc}{erfc}
\DeclareMathOperator{\spec}{spec}
\DeclareMathOperator{\TR}{TR}
\DeclareMathOperator{\Tr}{Tr}
\DeclareMathOperator{\tr}{tr}

\DeclareMathOperator{\ind}{index}
\DeclareMathOperator{\End}{End}

\DeclareMathOperator{\Spin}{Spin}
\DeclareMathOperator{\supp}{supp}

\DeclareMathOperator{\HS}{HS}

\newcommand{\tilTR}{\widetilde{\TR}}
\newcommand{\Spinc}{\Spin^c}

\newcommand{\C}{\ensuremath{\mathbb{C}}}

\newcommand{\N}{\ensuremath{\mathbb{N}}}

\newcommand{\R}{\ensuremath{\mathbb{R}}}

\newcommand{\Z}{\ensuremath{\mathbb{Z}}}

\newcommand{\Ca}[1]{\ensuremath{\mathcal{#1}}}
\newcommand{\cA}{\Ca{A}}
\newcommand{\cB}{\Ca{B}}

\newcommand{\cK}{\ensuremath{\mathcal{K}}}

\newcommand{\cM}{\Ca{M}}
\newcommand{\cN}{\Ca{N}}

\newcommand{\cS}{\ensuremath{\mathcal{S}}}

\newcommand{\APS}{APS}

\newcommand{\ba}{\begin{eqnarray}}
   \newcommand{\na}{\end{eqnarray}}

\newcommand{\beq}[1]{\begin{equation} \label{#1}}
\newcommand{\eeq}{\end{equation}}

\numberwithin{equation}{section}

\begin{document}


 \title{An equivariant Atiyah--Patodi--Singer index theorem for proper actions II: the $K$-theoretic index}
\author{Peter Hochs,\footnote{
University of Adelaide, Radboud University, \texttt{peter.hochs@adelaide.edu.au}} {}
Bai-Ling Wang\footnote{Australian National University,
\texttt{bai-ling.wang@anu.edu.au}} {} and
Hang Wang\footnote{East China Normal University, \texttt{wanghang@math.ecnu.edu.cn}}}

\maketitle

\begin{abstract}
Consider a proper, isometric action by a unimodular locally compact group $G$ on a Riemannian manifold $M$ with boundary, such that $M/G$ is compact. Then an equivariant Dirac-type operator $D$ on $M$ under a suitable boundary condition has an equivariant index $\ind_G(D)$ in
 the $K$-theory of the reduced group $C^*$-algebra $C^*_rG$ of $G$. This is a common generalisation of the Baum--Connes analytic assembly map and the (equivariant) Atiyah--Patodi--Singer index. 
 In part I of this series, a numerical index $\ind_g(D)$ was defined for an element $g \in G$, 
 in terms of a parametrix of $D$  and a trace associated to $g$. An 
 Atiyah--Patodi--Singer type index formula was obtained for this index.
 In this paper, we show that, under certain conditions, 
 \[
 \tau_g(\ind_G(D)) = \ind_g(D),
 \]
 for a trace $\tau_g$ defined by the orbital integral over the conjugacy class of $g$.
This implies that the index theorem from part I yields information about the $K$-theoretic index $\ind_G(D)$. It also shows that $\ind_g(D)$ is a homotopy-invariant quantity.
 \end{abstract}

\tableofcontents

\section{Introduction}

This paper is about a $K$-theoretic index defined for Dirac operators on manifolds with boundary, equivariant with respect to proper, cocompact actions by locally compact groups. It is a companion paper to part I \cite{HWW} of this series of two papers, in which numerical indices were defined for such operators, and an index formula was proved for those indices.
The main result in this paper is Theorem \ref{thm taug indG}, stating that, under certain conditions, numerical invariants extracted from the $K$-theoretic index via orbital integral traces equal  the indices from \cite{HWW}.
In that way, the index formula from \cite{HWW} applies to $K$-theoretic index as well.


Consider a unimodular, locally compact group $G$ acting properly and isometrically on a Riemannian manifold $M$, with boundary $N$, such that $M/G$ is compact. Let $D$ be a $G$-equivariant Dirac-type operator on a $G$-equivariant, $\Z_2$-graded Hermitian vector bundle $E  = E_+ \oplus E_- \to M$. Suppose that all structures have a product form near $N$. In particular, suppose that near $N$, the restriction of $D$ to sections of $E_+$ equals
\beq{eq DU intro}
\sigma \Bigl(-\frac{\partial}{\partial u} + D_N\Bigr),
\eeq
where $\sigma\colon E_+|_N \to E_-|_N$ is an equivariant vector bundle isomorphism, $u$ is the coordinate in $(0,1]$ in a neighbourhood of $N$ equivariantly isometric to $N \times(0,1]$, and
$D_N$ is a Dirac operator on $E_+|_N$.

We initially assume $D_N$ to be invertible, and later show how to weaken this assumption to $0$ being isolated in the spectrum of $D_N$. If $D_N$ is invertible, then we use the construction of an index
\beq{eq index intro}
\ind_G(D) \in K_0(C^*_rG)
\eeq
from \cite{GHM}, where $C^*_rG$ is the reduced group $C^*$-algebra of $G$. This index was defined in \cite{GHM} in a more general setting, and applied to, for example, Callias-type operators and positive scalar curvature \cite{GHM3} and the quantisation commutes with reduction problem \cite{GHM2}.

%
To extract relevant numbers from this $K$-theoretic index, we apply traces defined by \emph{orbital integrals}. Let $g \in G$, let $Z_g$ be its centraliser, and suppose that $G/Z_g$ has a $G$-invariant measure $d(hZ_g)$. Then the orbital integral with respect to $g$ of a function $f \in C_c(G)$ is the number
\beq{eq taug intro}
\tau_g(f) := \int_{G/Z_g} f(hgh^{-1})\, d(hZ_g).
\eeq
If the integral on the right hand side
 converges absolutely for all $f$ in a dense subalgebra  $\cA \subset C^*_rG$, closed under holomorphic functional calculus, then this defines a trace $\tau_g$ on $\cA$. That trace induces
\beq{eq taug Kthry}
\tau_g\colon K_0(C^*_rG) = K_0(\cA) \to \C.
\eeq

Orbital integrals for semisimple Lie groups are fundamental to Harish-Chandra's development of harmonic analysis on such groups. They also play an important role in Bismut's work on hypo-elliptic Laplacians \cite{BismutHypo}.
The map \eqref{eq taug Kthry} on $K$-theory is given by evaluating characters at $g$ if $G$ is compact. One can also use \eqref{eq taug Kthry} to recover the values at elliptic elements $g$ of characters of discrete series representation of semisimple groups \cite{HW2}. This
was used to link index theory to representation theory in \cite{HW2}. Higher cyclic cocycles generalising orbital integrals and capturing all information about classes in $K_*(C^*_rG)$ were developed by Song and Tang \cite{ST19}. 

For discrete groups, where they are sums over conjugacy classes, orbital integrals and the map \eqref{eq taug Kthry} have found various applications to geometry and topology in recent years, see for example \cite{KS17, WangWang, Weinberger-Yu, XieYu}.

Applying \eqref{eq taug Kthry} to \eqref{eq index intro} yields the number
\beq{eq taug index}
\tau_g(\ind_G(D)),
\eeq
which is the main object of interest in this paper.
%
The index \eqref{eq index intro} and the number \eqref{eq taug index} generalise various earlier indices. 
\begin{itemize}
\item
If $N = \emptyset$, then \eqref{eq index intro} is the image of $D$ under the Baum--Connes analytic assembly map \cite{Connes94}, see Corollary 4.3 in \cite{GHM}. That is the most natural and widely-used generalisation of the classical equivariant index to proper, cocompact actions. It has been applied to various problems in geometry and topology, such as questions about positive scalar curvature and the Novikov conjecture.  In this context, the number \eqref{eq taug index} was shown to be relevant to representation theory, orbifold geometry and trace formulas \cite{HST20,HW2, HW4, WangWang}.
\item
If $M$ and $G$ are compact, then \eqref{eq index intro} becomes the equivariant {\APS} index used in \cite{Donnelly79}, and \eqref{eq taug index} is the evaluation of that index at $g$.
(See Lemma 2.9 in \cite{HWW}.)
 If $G$ is trivial, then this index  reduces to the usual {\APS} index. 
\item
In the case where $M/G$ is a compact manifold with boundary, $M$ is its universal cover, and $G$ is its fundamental group, the number $\tau_e(\ind_G(D))$ is the index used by Ramachandran in \cite{Ramachandran}, see Remark \ref{rem ramachandran}. In this setting, the index \eqref{eq index intro} was introduced in Section 3 of \cite{XieYu14}. Indices with values in $K_*(C^*_rG)$ in this setting were also defined in \cite{Leichtnam97, Leichtnam98, Leichtnam00, Lott92}, via operators on Hilbert $C^*_rG$-modules and in \cite{Piazza14} in terms of Roe algebras. We expect these to be special cases of \eqref{eq index intro}, because they generalise the case of manifolds without boundary \cite{Miscenko}, a special case of the Baum--Connes assembly map; see also for example Proposition 2.4 in \cite{Piazza14}. 
\end{itemize}
These special cases suggest that the index \eqref{eq index intro} and the number \eqref{eq taug index} are natural objects to study. They generalise to the case where $0$ is isolated in the spectrum of $D_N$, as discussed  in Section \ref{sec DN not inv}.

In \cite{HWW}, the notion of a \emph{$g$-Fredholm operator} was introduced. Such operators have a numerical \emph{$g$-index}, defined in terms of a parametrix of the operator and a trace related to $\tau_g$. It was shown that for several classes of groups and actions, the Dirac operator $D$ on the manifold with boundary $M$ is $g$-Fredholm, and hence has a $g$-index, denoted by $\ind_g(D)$. An index formula was proved for this index. In the case where $D$ is a twisted $\Spinc$-Dirac operator, this index formula takes the form
\[
\ind_g(D)= \int_{M^g} \chi_g^2 \frac{\hat A(M^g) e^{c_1(L|_{M^g})/2} \tr(ge^{-R_{V}|_{M^g}/2\pi i}) }{\det(1-g e^{-R_{\cN}/2\pi i})^{1/2}} - \frac{1}{2}\eta_g(D_N).
\]
The first term in the right hand side is a direct generalisation of the right hand side of the Atiyah--Segal--Singer fixed point formula \cite{Atiyah68,ASIII, BGV}. The number $\eta_g(D_N)$ is a \emph{delocalised $\eta$-invariant}. These were first constructed by Lott \cite{Lott92, Lott99}.

The main result in this paper, Theorem \ref{thm taug indG}, states that, under certain conditions,
\[
\ind_g(D) = \tau_g(\ind_G(D)).
\]
This links the index $\ind_g(D)$ to $K$-theory, and allows us to apply the index formula from \cite{HWW} to the number \eqref{eq taug index}. This generalises the index theorems in \cite{APS1,Donnelly79,Ramachandran}, for example. Furthermore, homotopy invariance of $\ind_G$ implies homotopy invariance of $\ind_g$ in these cases.

\subsection*{Outline of this paper}

The index \eqref{eq index intro}, and the Roe algebras needed to define it, are introduced in Section \ref{sec prelim}. There we also recall the definition of the index $\ind_g$ from \cite{HWW}, and state the main result, Theorem \ref{thm taug indG}.

We prepare for the proof of  Theorem \ref{thm taug indG} in Section \ref{sec param}, by introducing a parametrix for the operator $D$, and discussing some properties of the $g$-trace and of heat kernels. Then we prove the two main steps in the proof of Theorem \ref{thm taug indG} in Sections \ref{sec Sj2 g tr cl} and \ref{sec trace index}, Propositions \ref{prop Sj2 g trace cl} and \ref{prop coarse index}. Combining these with a last extra step, Proposition \ref{prop Sj2}, we obtain a proof of Theorem \ref{thm taug indG}. In Section \ref{sec DN not inv}, we show how to weaken the assumption that the boundary Dirac operator $D_N$ in \eqref{eq DU intro} is invertible, to the assumption that $0$ is isolated in its spectrum.

\subsection*{Acknowledgements}

HW was supported by the Australian Research Council, through Discovery Early Career Researcher Award
DE160100525, by the Shanghai Rising-Star Program grant 19QA1403200, and by NSFC grant 11801178. PH thanks East China Normal University for funding a visit there in 2018.

\section{Preliminaries and results} \label{sec prelim}

For a proper, cocompact action by a general locally compact group $G$, the most widely-used equivariant index of equivariant elliptic operators is the Baum--Connes analytic assembly map \cite{Connes94}. (Here an action is called cocompact if its quotient is compact.) This is a generalisation of the usual equivariant index in the compact case, and takes values in $K_*(C^*_rG)$, the $K$-theory of the reduced group $C^*$-algebra of $G$.
In \cite{GHM},
a generalisation of the assembly map was constructed and studied, which applies to possibly non-cocompact actions, as long as the operator it is applied to is invertible outside a cocompact set in the appropriate sense.
This index
also generalises the Gromov--Lawson index \cite{Gromov83}, an equivariant index of Callias-type operators \cite{Guo18},  the (equivariant) {\APS} index on manifolds with boundary \cite{APS1, Donnelly}, and the index used by Ramachandran for manifolds with boundary \cite{Ramachandran}. This index is an equivariant version of the localised coarse index of Roe \cite{Roe16}, for actions by arbitrary locally compact groups.
For actions by fundamental groups of manifolds on their universal covers, this index was constructed in \cite{XieYu14}.
%

We  briefly review the construction of the index in \cite{GHM} in Subsection \ref{sec def index}, in the case we need here. This involves localised Roe algebras, which we discuss in Subsection \ref{sec Roe alg}. The index takes values in the $K$-theory of the reduced $C^*$-algebra of the group acting. Using traces on subalgebras of this algebra defined by orbital integrals, defined in Subsection \ref{sec orbital}, we extract numbers from that index. The main result in this paper is Theorem \ref{thm taug indG}, which states that, under certain conditions, those numbers equal the numbers for which an index formula was proved in \cite{HWW}.


\subsection{The localised equivariant Roe algebra} \label{sec Roe alg}

Let $(X, d)$ be a metric space in which all closed balls are compact. Let $G$ be a locally compact, unimodular group acting properly and isometrically on $X$. Let $Z \subset X$ be a nonempty, closed, $G$-invariant subset such that $Z/G$ is compact.
Fix a $G$-invariant Borel measure on $X$ for which every open set has positive measure. Let $E \to X$ be a $G$-equivariant Hermitian vector bundle.

The Hilbert space $L^2(E)$ of square-integrable sections of $E$ has a natural unitary representation of $G$, and an action by $C_0(X)$ given by pointwise multiplication of sections by functions. In this sense, it is a \emph{$G$-equivariant $C_0(X)$-module}. We will not define the various types of such modules here, but always work with concrete examples. Apart from $L^2(E)$, we will also use the module $L^2(E) \otimes L^2(G)$, where $G$ acts diagonally (acting on $L^2(G)$ via the left regular representation), and where $C_0(X)$ acts on the factor $L^2(E)$ via pointwise multiplication. If $X/G$ is compact, then $L^2(E) \otimes L^2(G)$ is an \emph{admissible} equivariant $C_0(X)$-module, under the non-essential assumption that  either $X/G$ or $G/K$, for a maximal compact subgroup $K<G$, is infinite. See Theorem 2.7 in \cite{GHM}. This type of $C_0(X)$-module is central to the constructions in \cite{GHM}.

We denote the algebra of $G$-equivariant bounded operators on a Hilbert space $H$ with a unitary representation of $G$ by $\cB(H)^G$.
\begin{definition} \label{def Roe}
Let $T \in \cB(L^2(E)\otimes L^2(G))$. Then $T$ is \emph{locally compact} if the operators $Tf$ and $fT$ are compact for all $f \in C_0(X)$. The operator $T$ has \emph{finite propagation} if there is a number $r>0$ such that for all $f_1, f_2 \in C_0(X)$ whose supports are further than $r$ apart, we have $f_1 T f_2 = 0$. Finally, $T$ is \emph{supported near $Z$} if there is an $r'>0$ such that for all $f \in C_0(X)$ whose support is further than $r'$ away from $Z$, the operators $Tf$ and $fT$ are zero.

The \emph{localised equivariant Roe algebra} of $X$ is the closure in $\cB(L^2(E)\otimes L^2(G))$ of the algebra of locally compact operators in $\cB(L^2(E)\otimes L^2(G))^G$ with finite propagation, supported near $Z$. It is denoted by $C^*(X)^G_{\loc}$.
\end{definition}
The algebra $C^*(X)^G_{\loc}$ is independent of $Z$. And, assuming either $Z/G$ or $G/K$  is an infinite set,
\beq{eq iso Roe group}
C^*(X)^G_{\loc} \cong C^*_rG \otimes \cK,
\eeq
where $C^*_rG$ is the reduced group $C^*$-algebra of $G$, and $\cK$ is the algebra of compact operators on a separable, infinite-dimensional Hilbert space. See (5) in \cite{GHM}. (If $Z/G$ and $G/K$ are both finite, then \eqref{eq iso Roe group} still holds with $\cK$ replaced by a matrix algebra.) This equality implies that $C^*(X)^G_{\loc}$ is also independent of $E$. (In fact, it is independent of the choice of a more general kind of admissible module.)

\begin{remark}
There is no reason a priori to assume that $Z/G$ is compact. The resulting localised Roe algebra will then depend on $Z$. We always assume that $Z/G$ is compact, so that we have the isomorphism \eqref{eq iso Roe group}, and we can apply the traces of Subsection \ref{sec orbital} to classes in the $K$-theory of $C^*(X)^G_{\loc}$.
\end{remark}

We will also use a version of the localised equivariant Roe algebra defined with respect to the $C_0(X)$-module $L^2(E)$, instead of $L^2(E)\otimes L^2(G)$. This is defined exactly as in Definition \ref{def Roe}, with $L^2(E)\otimes L^2(G)$ replaced by $L^2(E)$ everywhere. The resulting algebra is denoted by $C^*(X; L^2(E))^G_{\loc}$. This algebra is less canonical than $C^*(X)^G_{\loc}$, and is not stably isomorphic to $C^*_rG$ in general. If $X/G$ itself is compact, then we omit the subscript $\loc$, since being supported near $Z$ then becomes a vacuous condition.

\subsection{The localised equivariant coarse index} \label{sec def index}

Suppose, from now on, that $X = M$ is a complete Riemannian manifold, and $E$ is a smooth, $\Z_2$-graded, $G$-equivariant, Hermitian vector bundle. Let $D$ be an elliptic, odd-graded, essentially self-adjoint, first order differential operator on $E$. Suppose that 
\beq{eq D2 pos}
D^2 \geq c 
\eeq
on $M \setminus Z$, for a positive constant $c$. Let $b \in C(\R)$ be an odd function such that $b(x) = 1$ for all $x \geq c$. Lemma 2.3 in \cite{Roe16} states that $b(D)^2-1 \in C^*(X; L^2(E))^G_{\loc}$. By Lemma 2.1 in \cite{Roe16}, the operator $b(D)$ lies in the multiplier algebra $\cM(C^*(X; L^2(E))^G_{\loc})$ of $C^*(X; L^2(E))^G_{\loc}$. Hence the restriction of $b(D)$ to even-graded sections defines a class
\[
[b(D)] \in K_1 \bigl(\cM(C^*(X; L^2(E))^G_{\loc})/C^*(X; L^2(E))^G_{\loc} \bigr).
\]
Let
\[
\partial\colon K_1 \bigl(\cM(C^*(X; L^2(E))^G_{\loc})/C^*(X; L^2(E))^G_{\loc} \bigr) \to K_0 (C^*(X; L^2(E))^G_{\loc})
\]
be the boundary map in the six-term exact sequence associated to the ideal $C^*(X; L^2(E))^G_{\loc}$ of $\cM(C^*(X; L^2(E))^G_{\loc})$. We set
\beq{eq def L2E ind}
\ind_G^{L^2(E)}(D) := \partial[b(D)] \quad \in K_0 (C^*(X; L^2(E))^G_{\loc}).
\eeq

To obtain an index in $K_0(C^*_rG)$, let $\chi \in C^{\infty}(M)$ be a cutoff function, in the sense that it is nonnegative, its support has compact intersections with all $G$-orbits, and that for all $m \in M$,
\beq{eq def cutoff}
\int_G \chi(gm)^2\, dg = 1.
\eeq
The map
\beq{eq def j}
j\colon L^2(E)\to L^2(E) \otimes L^2(G),
\eeq
given by
\[
(j(s))(m,g) = \chi(g^{-1}m)s(m),
\]
for $s \in L^2(E)$, $m \in M$ and $g \in G$, is a $G$-equivariant, isometric embedding. Let
\beq{eq def plus 0}
\oplus 0 \colon C^*(X; L^2(E))^G_{\loc} \to C^*(X)^G_{\loc}
\eeq
be given by mapping operators on $L^2(E)$ to operators on $j(L^2(E))$ by conjugation with $j$, and extending them by zero on the orthogonal complement of $j(L^2(E))$. We denote the map on $K$-theory induced by $\oplus 0$ be the same symbol.
\begin{definition} \label{def index}
The \emph{localised equivariant coarse index} of $D$ is
\[
\ind_G(D) := \ind_G^{L^2(E)}(D) \oplus 0 \quad \in K_0(C^*_rG).
\]
\end{definition}

\begin{remark}
In \cite{GHM}, the localised equivariant coarse index is defined slightly differently from Definition \ref{def index}, but also in terms of $j$. The two definitions agree by (13) in \cite{GHM}. In that paper, a version for ungraded vector bundles, with values in odd $K$-theory, is also defined. An illustration of how (representation theoretic) information that may not be encoded by $\ind_G^{L^2(E)}(D)$ is recovered through the map $\oplus 0$ is Example 3.8 in \cite{GHM}.
\end{remark}

The index of Definition \ref{def index} simultaneously generalises various other indices; some are mentioned in the introduction.
For example, if $M/G$ is compact, then it reduces to the analytic assembly map from the Baum--Connes conjecture \cite{Connes94}. See Section 3.5  in \cite{GHM} for other special cases. In this paper, we apply the index to manifolds with boundary, to generalise the {\APS} index and its  generalisations in \cite{APS1, Donnelly, Ramachandran}.

\subsection{Orbital integrals} \label{sec orbital}

Fix an element $g \in G$. Let $Z_g <G$ be its centraliser. Suppose that $G/Z_g$ has a $G$-invariant measure $d(hZ_g)$ such that for all $f \in C_c(G)$,
\[
\int_G f(h)\, dh = \int_{G/Z_g} \int_{Z_g} f(hz)\, dz\, d(hZ_g),
\]
for fixed Haar measures $dh$ on $G$ and $dz$ on $Z_g$. (This is the case, for example, if $G$ is discrete, or if $G$ is real semisimple and $g$ is a semisimple element.)

The \emph{orbital integral} of a function $f \in C_c(G)$ is
\[
\tau_g(f) := \int_{G/Z_g} f(hgh^{-1})\, d(hZ_g).
\]
We assume that there is a dense subalgebra $\cA \subset C^*_rG$, closed under holomorphic functional calculus, such that $\tau_g$ extends to a continuous linear functional on $\cA$. 
 Then it defines a trace on $\cA$. Existence of $\cA$ is a nontrivial question;  positive answers for discrete and semisimple groups were given in \cite{KS17} and \cite{HW2}, respectively.
The trace $\tau_g$ on $\cA$ defines a map
\[
\tau_g\colon K_0(C^*_rG) = K_0(\cA) \to \C.
\]
Consider the setting of Subsection \ref{sec def index}. Then we have  the number
\[
\tau_g(\ind_G(D)).
\]

In part I \cite{HWW}, we used a trace related to $\tau_g$ to define the notion of a $g$-Fredholm operator, and the $g$-index of such operators. We briefly recall the definitions here.

%
%
Let $\chi \in C^{\infty}(M)$ be a cutoff function for the action, as in \eqref{eq def cutoff}.
Consider the bundle
\[
\End(E):= E \boxtimes E^* \to M \times M.
\]
\begin{definition} \label{def g trace class}
A section $\kappa \in \Gamma^{\infty}(\End(E))^G$ is \emph{$g$-trace class} if the integral
\beq{eq def Trg}
\int_{G/Z_g} \int_M \chi(hgh^{-1}m)^2 \tr(hgh^{-1} \kappa(hg^{-1}h^{-1} m, m))\, dm\, d(hZ_g)
\eeq
converges absolutely. Then the value of this integral
is the \emph{$g$-trace} of $\kappa$, denoted by $\Tr_g(\kappa)$. If $T$ is a bounded, $G$-equivariant operator on $L^2(E)$, with a $g$-trace class Schwartz kernel $\kappa$, then we say that $T$ is $g$-trace class, and define $\Tr_g(T) := \Tr_g(\kappa)$.
\end{definition}

\begin{definition}\label{def g Fredholm}
Let $D$ be a $G$-equivariant, elliptic differential operator on $E$, odd with respect to a $\Z_2$-grading on $E$. Let $D_+$ be its restriction to even-graded sections. Then $D$ is \emph{$g$-Fredholm} if $D_+$ has a parametrix $R$ such that the operators
\beq{eq Sj general}
\begin{split}
S_0 &:=1-RD_+;\\
S_1 &:=1-D_+R;\\
\end{split}
\eeq
are $g$-trace class.

The \emph{$g$-index} of a $g$-Fredholm operator $D$ is the number
\beq{eq def g index}
\ind_g(D) := \Tr_g(S_0) - \Tr_g(S_1),
\eeq
with $S_0$ and $S_1$ as in \eqref{eq Sj general}.
\end{definition}
The $g$-index is independent of the parametrix $R$ by Lemma 2.4 in \cite{HWW}.

\subsection{Manifolds with boundary} \label{sec bdry}

We now specialise to the case we are interested in in this paper. The setting is the same as in Subsection 2.2 in \cite{HWW}.

Slightly changing notation from the previous subsections, we let $M$ be a Riemannian manifold with boundary $N$. We still suppose that $G$ acts properly and isometrically on $M$, preserving $N$, such that $M/G$ is compact. We assume that a $G$-invariant neighbourhood $U$ of $N$ is $G$-equivariantly isometric to a product $N \times (0,\delta]$, for a $\delta > 0$. To simplify notation, we assume that $\delta = 1$; the case for general $\delta$ is entirely analogous.

As before, let  $E = E_+ \oplus E_- \to M$  be a $\Z_2$-graded $G$-equivariant, Hermitian vector bundle. We assume that $E$ is a \emph{Clifford module}, in the sense that there is a $G$-equivariant vector bundle homomorphism, the Clifford action, from the Clifford bundle of $TM$ to the endomorphism bundle of $E$, mapping odd-graded elements of the Clifford bundle to odd-graded endomorphisms. We also assume that there is a $G$-equivariant isomorphism of Clifford modules $E|_U \cong E|_N \times (0,1]$.

Let $D$ be a Dirac-type operator on $E$; i.e.\ the principal symbol of $D$ is given by the Clifford action. Let $D_+$ be the restriction of $D$ to sections of $E_+$. Suppose that
\beq{eq D+U}
D_+|_U = \sigma \Bigl(-\frac{\partial}{\partial u} + D_N\Bigr),
\eeq
where $\sigma\colon E_+|_N \to  E_-|_N$ is a $G$-equivariant vector bundle isomorphism, $u$ is the coordinate in the factor $(0,1]$ in $U = N \times (0,1]$, and $D_N$ is an (ungraded) Dirac-type operator on $E_+|_N$. We initially assume that $D_N$ is \emph{invertible}, and show how to remove this assumption in Section \ref{sec DN not inv}.

Consider the cylinder $C := N \times [0,\infty)$, equipped with the product of the mertic on $M$ restricted to $N$, and the Euclidean metric. Because the metric, group action, Clifford module and Dirac operator have a product form on $U$, all these structures extend to $C$. We denote 
 the extension of $D$ to $C$ by $D_C$. We form the complete manifold
\[
\hat M := (M \sqcup C)/\sim,
\]
where $m \sim (n,u)$ if $m  = (n,u)\in U = N \times (0,1]$. Let $\hat E \to \hat M$ and $\hat D$ be the extensions of $E$ and $D$ to $\hat M$, respectively, obtained by gluing the relevant objects on $M$ and $C$ together along $U$.

Since $D_N$ is invertible, there is a $c>0$ such that 
\beq{eq DN pos c}
D_N^2 \geq c. 
\eeq
Hence also
 $D_C^2 \geq c$. In other words, $\hat D^2 \geq c$ outside the cocompact set $M$, so that Definition \ref{def index} applies to $\hat D$. This gives us the localised coarse index
\beq{eq APS index}
\ind_G(\hat D) \in K_0(C^*_rG),
\eeq
which is the main object of study in this paper. Our goal is to give a topological expression for the number $\tau_g(\ind_G(\hat D))$.

The index \eqref{eq APS index} and the number $\tau_g(\ind_G(\hat D))$ simultaneously generalise several widely-used indices, as mentioned in the introduction. The index generalises to a case where $D_N$ is not invertible, as discussed in Section \ref{sec DN not inv}.

\subsection{The main result}\label{sec result}

We assume that two  heat kernels associated to $D$ satisfy standard Gaussian decay properties. Let $\tilde D$ be the extension of $D$ to the double $\tilde M$ of $M$.
 Let $\kappa_t$ be the smooth kernel of either $e^{-t\tilde D^2}$, $\tilde De^{-t\tilde D^2}$, $e^{-tD_C^2}$ or $D_C e^{-tD_C^2}$. we assume that for all $t_0>0$, there are $b_1,b_2,b_3>0$ such that for all $t \in (0,t_0]$ and all $m,m'$ in the relevant manifold $\tilde M$ or $N \times \R$,
\beq{eq heat kernel decay}
\|\kappa_t(m,m')\| \leq b_1 t^{-b_2}e^{-b_3 d(m,m')^2/t},
\eeq
where $d$ is the Riemannian distance. Estimates of this type were proved in many places. A classical result is the one by Chen--Li--Yau \cite{CLY81} for the scalar Laplacian. Note that any bounded geometry-type conditions are automatically satisfied in our setting, because $N/G$ and $\tilde M/G$ are compact.

In the case where $G = \Gamma$ is discrete and finitely generated, 
let $l$ be a word length function on $\Gamma$ with respect to a fixed, finite, symmetric, generating set. Because $\Gamma$ is finitely generated, there are $C,k>0$ such that for all $n \in \N$,
\beq{eq growth Gamma}
\#\{\gamma \in \Gamma; l(\gamma) = n\} \leq Ce^{k n}.
\eeq
Fix $m_0 \in M$.
By the Svarc--Milnor lemma, there are $a_1, a_2>0$ such that for all  $\gamma \in \Gamma$,
\beq{eq SM}
d(\gamma m_0, m_0) \geq a_1l(\gamma) - a_2.
\eeq
Let $c$ be as in \eqref{eq DN pos c}.
\begin{theorem}\label{thm taug indG}
Suppose that  $\hat D$ is $g$-Fredholm, and that the heat kernel decay \eqref{eq heat kernel decay} holds for the operators mentioned. Suppose that an algebra $\cA$ as in Subsection \ref{sec orbital} exists.
If either
\begin{itemize}
\item[(a)] 
$G/Z_g$ is compact; or
\item[(b)] 
$G = \Gamma$ is discrete and finitely generated, and \eqref{eq growth Gamma} holds for a $k<\frac{2 a_1 \sqrt{c}}{3}$,
\end{itemize}
then
\[
\tau_g(\ind_G(\hat D)) = \ind_g(\hat D).
\]
\end{theorem}
Conditions for $\hat D$ to be $g$-Fredholm were given in Theorem 2.11 and Corollaries 2.13, 2.16 and 2.18 in \cite{HWW}.

%

\begin{remark}
The growth condition on $\Gamma$ in part (b) of Theorem \ref{thm taug indG} holds in particular if $\Gamma$ has slower than exponential growth. In general, the condition depends on $D$, $\Gamma$ and the group action. The factor $2/3$ in the bound $\frac{2 a_1 \sqrt{c}}{3}$ may be increased to any number smaller than $1$. This can be achieved if we replace the factors $1/3$ on the right hand sides of \eqref{eq nu} by other factors smaller than $1/2$.
\end{remark}

The first corollary of Theorem \ref{thm taug indG} is homotopy invariance of $\ind_g(\hat D)$. This follows by homotopy invariance of $\ind_G(\hat D)$.
\begin{corollary}
In the setting of Theorem \ref{thm taug indG},
the number $\ind_g(\hat D)$ is a homotopy-invariant property of $\hat D$.
\end{corollary}

Combining Theorem \ref{thm taug indG} with Corollaries 2.13 and 2.16 in \cite{HWW}, we obtain an index formula for $\tau_g(\ind_G(\hat D))$.
\begin{corollary} \label{cor index taug}
Let $D$ be a twisted $\Spinc$-Dirac operator. Suppose that  the heat kernel decay \eqref{eq heat kernel decay} holds for the operators mentioned.
Suppose that either
\begin{itemize}
\item $g=e$, or
\item $G = \Gamma$ is discrete and finitely generated, \eqref{eq growth Gamma} holds for a $k<\frac{2 a_1 \sqrt{c}}{3}$, and $(g)$ has polynomial growth.
\end{itemize}
Then 
\beq{eq Spinc Dirac}
\tau_g(\ind_G(\hat D))= \int_{M^g} \chi_g^2 \frac{\hat A(M^g) e^{c_1(L|_{M^g})/2} \tr(ge^{-R_{V}|_{M^g}/2\pi i}) }{\det(1-g e^{-R_{\cN}/2\pi i})^{1/2}} - \frac{1}{2}\eta_g(D_N).
\eeq
\end{corollary}
Notation is as in \cite{HWW}; the integrand on the right hand side is the Atiyah--Segal--Singer integrand \cite{Atiyah68,ASIII,BGV} times a cutoff function $\chi_g^2$, and $\eta_g(D_N)$ is the delocalised $\eta$-invariant of $D_N$, as in \cite{Lott92,Lott99} and Subsection 2.3 of \cite{HWW}.

Theorem \ref{thm taug indG}, combined with results from \cite{CWXY19}, also implies a version of Proposition 5.3 in \cite{XieYu} and Theorem 1.4 in \cite{CWXY19} in the case of fundamental groups of compact manifolds with boundary acting on their universal covers.
\begin{corollary}\label{cor free}
Suppose that $X$ is a compact Riemannian $\Spinc$-manifold with boundary, with a product structure near the boundary. Let $M$ be the universal cover of $X$, and let $N =\partial M$ as before.
 Let $G=\Gamma = \pi_1(X)$. Let $D$ be the lift to $M$ of a twisted $\Spinc$-Dirac operator on $X$. Let $g \in \Gamma$ be different from the identity element. If the constant $c$ such that $D_N^2 \geq c$ is large enough, then
\[
\tau_g(\ind_G(\hat D))=- \frac{1}{2}\eta_g(D_N).
\]
\end{corollary}
\begin{proof}
If the constant $c$ is large enough, then the delocalised $\eta$-invariant $\eta_g(D_N)$ converges by Theorem 1.1 in \cite{CWXY19}. Furthermore, condition (b) in Theorem \ref{thm taug indG} also holds if $c$ is large enough. Finally, $\ind_{\Gamma}(\hat D) \in K_0(C^*_r \Gamma)$ may be replaced by an index in $K_0(l^1(\Gamma))$; see for example Remark A.2 in \cite{XieYu}. The trace $\tau_g$ converges on $l^1(\Gamma)$ without growth conditions on the conjugacy class of $g$. So Theorem 2.7 applies in this setting, and so does the index formula in Corollary 2.17 in \cite{HWW}. The interior contribution now equals zero, because a nontrivial group element has no fixed points because the action is free.
\end{proof}

\begin{remark}
The index theorem in \cite{HWW} also applies to semisimple Lie groups. So it is a natural question if a version of Theorem \ref{thm taug indG} applies in that setting. We expect the techniques needed to prove this (particularly Proposition \ref{prop Sj2 g trace cl}) to be very different from the discrete case. We have not looked into the details so far.
\end{remark}


\begin{remark}\label{rem ramachandran}
The case of Theorem \ref{thm taug indG} where $g=e$, combined with Lemma 2.6 in \cite{HWW}, shows that $\tau_e(\ind_G(\hat D))$ generalises the index used by Ramachandran in \cite{Ramachandran}, and that Corollary \ref{cor index taug} generalises Ramachandran's index theorem for manifolds with boundary.
\end{remark}

\begin{remark}
%
Consider the setting of Corollary \ref{cor free}. Let $D_{X}$ be the twisted $\Spinc$-Dirac operator on $X$ that lifts to the operator $D$ on $M$.
As a consequence of Theorem 3.9 in \cite{GHM2}, where reduced group $C^*$-algebras and Roe algebras are replaced by maximal ones (one can also use $l^1(\Gamma)$), we have
\beq{eq APS quotient}
\sum_{(g)} \tau_g(\ind_\Gamma(\hat D)) = \ind(D_{X}),
\eeq
where the sum runs over all conjugacy classes $(g)$ in $\Gamma$, and the index on the right hand side is the {\APS} index of $D_{X}$. Since $\Gamma$ acts freely on $M$, Corollaries \ref{cor index taug} and \ref{cor free} imply that the left hand side of \eqref{eq APS quotient} equals
\[
\int_{M} \chi_e^2 {\hat A(\tilde M) e^{c_1(L)/2} \tr(ge^{-R_{\tilde V}/2\pi i}) } - \frac{1}{2} \sum_{(g)}\eta_g(D_N).
\]
The first term is exactly the interior contribution to the topological side of the {\APS} index  of $D_{X}$. We conclude that
\[
\eta(D_{Y}) = \sum_{(g)}\eta_g(D_N),
\]
where $D_{Y}$ is the Dirac operator on the boundary $Y = N/\Gamma$ of $X$ corresponding to $D_N$.
In other words, the delocalised $\eta$-invariants of $D_N$ are refinements of the $\eta$-invariant of $D_{Y}$. This remains true in a case where $D_N$ is not invertible, but there is a large enough gap in the spectrum of $D_N$ around zero. See Section \ref{sec DN not inv}.  (See (I.6) in \cite{Donnelly} for the  case where $G$ is finite.)

We expect Corollary \ref{cor index taug}, and its extension to non-invertible $D_N$, to refine Farsi's orbifold {\APS} index theorem (Theorem 4.1 in \cite{Farsi07}) in a similar way. 
\end{remark}

\section{A parametrix and properties of the $g$-trace} \label{sec param}

We prepare for the proof of Theorem \ref{thm taug indG} by introducing a specific parametrix for $\hat D$, and discussing some properties of the $g$-trace and of heat operators. We will use these things in Section \ref{sec Sj2 g tr cl} to prove that for the parametrix chosen, the squares of the remainder terms $S_j$ as in \eqref{eq Sj general} are $g$-trace class in the setting of Theorem \ref{thm taug indG}. 

\subsection{A parametrix} \label{sec Sj}

We will use a parametrix of $\hat D$ introduced in Subsection 5.1 of \cite{HWW}.
 Consider the setting of Subsection \ref{sec bdry}.
As before,  let $\tilde M$ be the double of $M$, and let $\tilde E = \tilde E_+ \oplus \tilde E_-$ and $\tilde D$ be the extensions of $E$ and $D$ to $\tilde M$, respectively.
More explicitly, as on page 55 of \cite{APS1}, $\tilde M$ is obtained from $M$ by gluing together a copy of $M$ and a copy of $M$ with reversed orientation, while $\tilde E$ is obtained by  gluing together a copy of $E$ and a copy of $E$ with reversed grading. To glue these copies of $E$ together along $N$, we use the isomorphism $\sigma$.
Let $\tilde D_{\pm}$ be the restrictions of $\tilde D$ to the sections of $\tilde E_{\pm}$.

Let $\psi_1\colon (0,\infty) \to [0,1]$ be a smooth function such that $\psi_1$ equals $1$ on $(0, \varepsilon)$ and $0$ on $(1-\varepsilon,\infty)$, for some $\varepsilon \in (0,1/2)$. Set $\psi_2 := 1-\psi_1$. Let $\varphi_1, \varphi_2\colon(0,\infty) \to [0,1]$ be smooth functions such that $\varphi_1$ equals $1$ on $(0,1-\varepsilon/2)$ and $0$ on $(1,\infty)$, while $\varphi_2$ equals $0$ on $(0,\varepsilon/4)$ and $1$  on $(\varepsilon/2,\infty)$.
Then $\varphi_j \psi_j = \psi_j$ for $j = 1,2$, and $\varphi_j'$ and $\psi_j$ have disjoint supports.

We pull back the functions $\varphi_j$ and $\psi_j$ to $C$ along the projection onto $(0,\infty)$, and extend these functions smoothly to $\hat M$ by setting $\psi_1$ and $\varphi_1$ equal to $1$ on $M\setminus U$,  and $\psi_2$ and $\varphi_2$ equal to $0$ on $M\setminus U$. We denote the resulting functions by the same symbols $\psi_j$ and $\varphi_j$. (No confusion is possible in what follows, because we will always use these symbols to denote the functions on $\hat M$.) We denote the derivatives of these functions in the $(0, \infty)$ directions by $\varphi_j'$ and $\psi_j'$, respectively. These derivatives are only defined and used on $N \times (0,\infty) \subset \hat M$.

Fix $t >0$, and consider the parametrix
\beq{eq def tilde Q}
\tilde Q := \frac{1 - e^{-t\tilde D_- \tilde D_+}}{\tilde D_- \tilde D_+} \tilde D_-
\eeq
of $\tilde D_+$. (The part without the last factor $\tilde D_-$ is formed via functional calculus, by an application of the function $x\mapsto \frac{1-e^{-tx}}{x}$ to $\tilde D_- \tilde D_+$; this does not require invertibility of $\tilde D_- \tilde D_+$.)

The extension of $D_C$ to the complete manifold $N \times \R$ is essentially self-adjoint and positive. Hence its self-adjoint closure is invertible.
Let $Q_C$ be the restriction to sections of $E_-$ of the inverse of that  closure.
We define 
\[
R := \varphi_1 \tilde Q \psi_1 + \varphi_2 Q_C \psi_2.
\]
Note that the operator $\tilde Q$ is well-defined on the supports of $\varphi_1$ and $\psi_1$, and that $Q_C$ is well-defined on the supports of $\varphi_2$ and $\psi_2$.
The following two operators play key roles in this paper.
\beq{eq def Sj}
\begin{split}
S_0 &:= 1-R\hat D_+;\\
S_1 &:= 1-\hat D_+ R.
\end{split}
\eeq

\subsection{Properties of $S_0$ and $S_1$}\label{sec prop Sj}

Consider the setting of Subsection \ref{sec result}.
 In addition to the parametrix $R$ and the remainder terms $S_0$ and $S_1$, we will also use
 the remainders
\beq{eq tilde Sj}
\begin{split}
\tilde S_{0}&:= 1-\tilde Q \tilde D_+ = e^{-t\tilde D_- \tilde D_+};\\
\tilde S_{1} &:= 1-\tilde D_+ \tilde Q=  e^{-t\tilde D_+ \tilde D_-}.
\end{split}
\eeq

We recall Lemmas 5.1 and 5.2 from \cite{HWW}.
\begin{lemma} \label{lem comp Sj}
We have
\beq{eq S0 S1}
\begin{split}
S_0 &= \varphi_1 \tilde S_{0} \psi_1 + \varphi_1 \tilde Q \sigma \psi_1' + \varphi_2 Q_C \sigma \psi_2';\\
S_1 &= \varphi_1 \tilde S_{1} \psi_1 - \varphi_1' \sigma \tilde Q \psi_1 - \varphi_2' \sigma Q_C \psi_2.
\end{split}
\eeq
\end{lemma}
\begin{lemma} \label{lem Sj smoothing}
The operators $S_0$ and $S_1$ have smooth kernels.
\end{lemma}

\begin{lemma} \label{lem S0 S1}
The operators $S_0$ and $S_1$  lie in $C^{*}(\hat M; L^2(\hat E))^G_{\loc}$.
\end{lemma}
\begin{proof}
The operators $S_0$ and $S_1$ have smooth kernels by Lemma \ref{lem Sj smoothing}. This implies that these operators are locally compact.

The operator $Q_C$ equals $b(D_C)$, where $b \in C_0(\R)$ satisfies $b(x) = 1/x$ for all $x \in \spec(D_N) \not \ni 0$. Hence, by Lemma 2.1 in \cite{Roe16}, $Q_C$ is a norm-limit of a sequence $(Q_{C, j})_{j=1}^{\infty}$ operators with finite propagation. Similarly, $\tilde Q$ is a norm-limit of operators with finite propagation. So $S_0$ and $S_1$ are norm-limits of operators with finite propagation.

Since $\varphi_2'$ and $\psi_2'$ are supported near $M$ and $Q_{C, j}$ has finite propagation, the operators $\varphi_2' \sigma Q_{C,j} \psi_2$ and  $\varphi_2 Q_{C,j} \sigma \psi_2'$ are supported near $M$. Hence $\varphi_2' \sigma Q_{C} \psi_2$ and $\varphi_2 Q_{C} \sigma \psi_2'$ are norm-limits of operators that are supported near $M$. The other terms on the right hand sides of \eqref{eq S0 S1} are supported near $M$ because $\varphi_1$ and $\psi_1$ are. So $S_0$ and $S_1$ are norm-limits of operators that are supported near $M$.
\end{proof}

\subsection{Properties of the $g$-trace}

We consider a general setting, where $E \to M$ is an equivariant, Hermitian vector bundle over a complete Riemannian metric with a  proper, isometric action by $G$. In Subsection \ref{sec pf Sj2 g trace cl}, we return to the setting of Subsection \ref{sec bdry}.

This trace property is Lemma 3.4 in \cite{HWW}.
\begin{lemma} \label{lem ST TS}
Let  $S$ and $T$ are $G$-equivariant operators on $\Gamma^{\infty}(E)$. Suppose that 
$S$ has a distributional kernel supported on the diagonal, and 
$T$ has a smooth kernel in $\Gamma^{\infty}(\End(E))^G$. If $ST$ and $TS$ are $g$-trace class, then they have the same $g$-trace.
\end{lemma}



\begin{lemma}
A section $\kappa \in \Gamma^{\infty}(\End(E))^G$ is $g$-trace class if and only if 
the integral
\beq{eq def Trg 2}
\int_{G/Z_g} \int_M \chi(m)^2 | \tr(hgh^{-1} \kappa(hg^{-1}h^{-1} m, m)) |\, dm\, d(hZ_g)
\eeq
converges. 
\end{lemma}
\begin{proof}
In \eqref{eq def Trg}, substituting $m' = hgh^{-1}m$,  using $G$-invariance of $\kappa$ and the trace property shows that \eqref{eq def Trg} equals
\begin{multline*}
\int_{G/Z_g} \int_M \chi(m')^2 | \tr(hgh^{-1} \kappa(hg^{-2}h^{-1}  m', hg^{-1}h^{-1}m')) | \, dm'\, d(hZ_g)
\\= \int_{G/Z_g} \int_M \chi(m')^2 | \tr( \kappa(hg^{-1}h^{-1}  m', m')hgh^{-1})|\, dm'\, d(hZ_g) \\
= \int_{G/Z_g} \int_M \chi(m')^2 | \tr( hgh^{-1}\kappa(hg^{-1}h^{-1}  m', m')) |\, dm'\, d(hZ_g).
\end{multline*}
\end{proof}

\begin{lemma}  \label{lem GZ cpt}
Let $\kappa \in \Gamma^{\infty}(\End(E))^{G}$ be such that there exists a cocompactly supported $\varphi \in C^{\infty}(M)^{G}$ such that either $\kappa = (\varphi \otimes 1) \kappa$ or
$\kappa = (1 \otimes \varphi) \kappa$. Suppose that $G/Z_g$ is compact. Then $\kappa$ is $g$-trace class.
\end{lemma}
\begin{proof}
We prove the case where $\kappa = (\varphi \otimes 1) \kappa$, the other case is analogous.
The integral \eqref{eq def Trg 2} then equals
\[
\int_{G/Z_g} \int_M \chi(m)^2 \varphi(m) | \tr(hgh^{-1} \kappa(hg^{-1}h^{-1} m, m)) |\, dm\, d(hZ_g).
\]
Because $G/Z_g$ and the support of $\chi^2 \varphi$ are compact, this integral converges.
\end{proof}
In the setting of Lemma \ref{lem GZ cpt}, if $\kappa^2$ is well-defined, then it has the same property as $\kappa$, so that it is also $g$-trace class.

\subsection{$G$-integrable kernels}

The composition of two $g$-trace class operators need not be $g$-trace class. The notion of $G$-integrability (or $\Gamma$-summability for discrete groups $\Gamma$) can be used to prove that such compositions are $g$-trace class under certain conditions.

\begin{definition} \label{def G intble}
A section $\kappa \in \Gamma^{\infty}(\End(E))^G$ is \emph{$G$-integrable} if 
for all $\varphi, \psi \in C^{\infty}_c(M)$,
the integral
\[
\int_G
\left(
\int_{M \times M} \varphi(m) \psi(m')
 \| x\kappa(x^{-1}m, m')  \|^2 \,dm\, dm'\right)^{1/2} \, dx
\]
converges. 
\end{definition}

\begin{lemma} \label{lem int G}
Let $\kappa, \lambda \in \Gamma^{\infty}(\End(E))^{G}$ be $G$-integrable, and such that there exist cocompactly supported $\varphi, \psi \in C^{\infty}(M)^{G}$ such that either $\kappa = (\varphi \otimes 1) \kappa$ and $\lambda = (\psi\otimes 1) \lambda$ or
$\kappa = (1 \otimes \varphi) \kappa$ and $\lambda = (1 \otimes \psi) \lambda$. Suppose that the composition $\kappa \lambda$ is a well-defined element of $\Gamma^{\infty}(\End(E))^{G}$. Then the integral
\beq{eq G intble}
\int_{G} \int_M  \chi(m)^2 | \tr(x (\kappa \lambda) (x^{-1} m, m)) | \, dm\, dx
\eeq
converges.
\end{lemma}
\begin{proof}
We prove the case where $\kappa = (\varphi \otimes 1) \kappa$ and $\lambda = (\psi\otimes 1) \lambda$, the other case is analogous.
In this situation, the integral \eqref{eq G intble} equals
\[
\int_{G} \int_M \chi(m)^2  \left| \int_{M} \varphi(m) \psi(m') \tr(x \kappa(x^{-1}m, m') \lambda (m', m)) \, dm'\, \right| dm\, dx.
\]
Inserting a factor $1 = \int_{G} \chi(ym')^2\, dy$ and substituting $m'' = ym'$, we find that this integral equals at most
\begin{multline*}
\int_{G} \int_M\int_{M} \int_G\chi(m)^2 \chi(ym')^2  \left|  \varphi(m)  \psi(m') \tr(x \kappa(x^{-1}m, m') \lambda (m', m)) \right| \, dy \, dm' \, dm\, dx \\
=\int_{G} \int_M \int_{M} \int_G \chi(m)^2 \chi(m'')^2  \left|  \varphi(m) \psi(m'')  \tr(x \kappa(x^{-1}m, y^{-1}m'') \lambda (y^{-1}m'', m)) \right| \, dy \, dm'' \, dm\, dx.
\end{multline*}
By Fubini's theorem, the integral on the right converges if and only if
\[
 \int_G \int_{G} 
\int_M \int_{M}  \chi(m)^2 \chi(m'')^2
\bigl|  \varphi(m) \psi(m'')  \tr(x \kappa(x^{-1}m, y^{-1}m'') \lambda (y^{-1}m'', m)) \bigr|
\, dm'' \, dm  \, dx\, dy 
 \]
converges. It is enough to consider the case where $\chi$, $\varphi$ and $\psi$ are nonnegative. Then
the latter integral
is at most equal to 
\[
 \int_G \int_{G} 
\int_M \int_{M}
\chi(m)^2 \varphi(m) \chi(m'')^2 \psi(m'')
 \| x \kappa(x^{-1}m, y^{-1}m'')\lambda (y^{-1}m'', m)) \|
 \, dm'' \, dm  \, dx\, dy.
\]
Using $G$-invariance of $\kappa$, subtituting $z = xy^{-1}$ for $x$ and applying the Cauchy--Schwartz inequality, we see that this  integral equals
\begin{multline} \label{eq int G intble}
\int_M \int_{M}
\chi(m)^2 \varphi(m) \chi(m'')^2 \psi(m'')
 \int_G \int_{G} 
 \| xy^{-1} \kappa(yx^{-1}m, m'') y\lambda (y^{-1}m'', m)) \|
 \, dx\, dy \, dm'' \, dm\\
  \leq 
\int_G \left(
\int_{M \times M} \chi(m)^2 \varphi(m) \chi(m'')^2 \psi(m'')
 \| z \kappa(z^{-1}m, m'')\|^2
 \right)^{1/2}\, dz \\
 \cdot
   \int_G \left(
\int_{M \times M} \chi(m)^2 \varphi(m) \chi(m'')^2 \psi(m'')
 \| y \lambda(y^{-1}m, m'')\|^2
 \right)^{1/2}\, dy.
%
%
%
\end{multline}
The right hand side converges by $G$-integrability of $\kappa$ and $\lambda$.
%
\end{proof}

%

If $G = \Gamma$ is discrete, then we will call a $G$-integrable smooth kernel \emph{$\Gamma$-summable}.
\begin{lemma} \label{lem gamma tr cl}
Suppose $G = \Gamma$ is discrete.
Let $\kappa, \lambda \in \Gamma^{\infty}(\End(E))^{\Gamma}$ be $\Gamma$-summable, and such that there exist cocompactly supported $\varphi, \psi \in C^{\infty}(M)^{\Gamma}$ such that either $\kappa = (\varphi \otimes 1) \kappa$ and $\lambda = (\psi\otimes 1) \lambda$ or
$\kappa = (1 \otimes \varphi) \kappa$ and $\lambda = (1 \otimes \psi) \lambda$. Suppose that the composition $\kappa \lambda$ is a well-defined element of $\Gamma^{\infty}(\End(E))^{\Gamma}$. Then $\kappa \lambda$ is $\gamma$-trace class for all $\gamma \in \Gamma$.
\end{lemma}
\begin{proof}
By Lemma \ref{lem int G}, 
\[
\sum_{\gamma' \in \Gamma} \int_M  \chi(m)^2 | \tr(\gamma' (\kappa \lambda) (\gamma'^{-1} m, m)) | \, dm
\]
converges. So the sum over the conjugacy class of $\gamma$ also converges, which is \eqref{eq def Trg 2} in this case.
\end{proof}

\subsection{Basic estimates for heat operators}
Let $D$ be a Dirac operator on $E \to M$.

\begin{lemma} \label{lem sup norm op norm 3}
Let $f \in \cS(\R)$. Let $r\geq0$. Consider  bounded endomorphisms $\Phi$ and $\Psi$ of $E$ whose supports are at least a distance $r$ apart.
 Then
\[
\|\Phi f(D) \Psi\|_{\cB(L^2(E))} \leq \frac{1}{2\pi} \|\Phi\| \|\Psi\| \int_{\R \setminus [-r,r]} |\hat f(\xi)|\, d\xi.
\]
\end{lemma}
\begin{proof}
For $D = \sqrt{-\Delta}$, with $\Delta$ the scalar Laplacian and $f$ even, this is Proposition 1.1 in \cite{CGT82}. The arguments apply directly to $D$: the claim follows from the decomposition
\[
f(D) = \frac{1}{2\pi} \int_{\R} \hat f(\lambda) e^{i\lambda D}\, d\lambda
\]
and the fact that $e^{i\lambda D}$ has propagation at most $|\lambda|$. See Propositions 10.3.5 and 10.3.1 in \cite{Higson00}, respectively.
%
%
\end{proof}
\begin{corollary} \label{cor hk decay}
In the setting of Lemma \ref{lem sup norm op norm 3}, for all $t>0$,
\[
\begin{split}
\|\Phi e^{-tD^2} \Psi\|_{\cB(L^2(E))} & \leq \frac{2}{ \sqrt{\pi}} \|\Phi\| \|\Psi\| e^{-\frac{r^2}{4t}} \\
\|\Phi D e^{-tD^2} \Psi\|_{\cB(L^2(E))} &\leq \frac{1}{\sqrt{\pi t} } \|\Phi\| \|\Psi\| e^{-\frac{r^2}{4t}}.
\end{split}
\]
\end{corollary}
\begin{proof}
Applying Lemma \ref{lem sup norm op norm 3} with $f(x) = e^{-tx^2}$, we obtain
\[
\begin{split}
\|\Phi e^{-tD^2} \Psi\|_{\cB(L^2(E))} &\leq  \frac{1}{\sqrt{\pi t}} \|\Phi\| \|\Psi\| \int_r^{\infty} 
e^{-\frac{\lambda^2}{4t}}\, d\lambda \\
&=  \frac{2}{\sqrt{\pi}} \|\Phi\| \|\Psi\| \erfc\left(\frac{r}{2\sqrt{t}}\right).
\end{split}
\]
The first inequality now follows form the inequality $\erfc(x) \leq e^{-x^2}$ for all $x>0$.

For the second inequality, we take $f(x) = xe^{-tx^2}$. Then Lemma \ref{lem sup norm op norm 3} yields
\[
\begin{split}
\|\Phi D e^{-tD^2} \Psi\|_{\cB(L^2(E))} & \leq \frac{1}{2\sqrt{\pi} t^{3/2}} \|\Phi\| \|\Psi\|  \int_r^{\infty} \lambda e^{-\frac{\lambda^2}{4t}}\, d\lambda \\
&=  \frac{1}{\sqrt{\pi t} } \|\Phi\| \|\Psi\| e^{-\frac{r^2}{4t}}.
\end{split}
\]
\end{proof}
\begin{lemma} \label{lem HS}
Suppose that $M$ has bounded geometry, and that the kernels of $e^{-tD^2}$ and $e^{-tD^2}D$ satisfy bounds of the type \eqref{eq heat kernel decay}.
The operators $e^{-tD^2} \varphi$ and $e^{-tD^2}D \varphi$ 
 are Hilbert--Schmidt operators for all $t>0$ and $\varphi \in C^{\infty}_c(M)$. 
\end{lemma}
\begin{proof}
Let $\kappa$ be the Schwartz kernel of either $e^{-tD^2} D$ or $e^{-tD^2}$. 
The bound  \eqref{eq heat kernel decay} means that $\kappa \varphi$ can be bounded by a Gaussian function. Since $M$ has bounded geometry, volumes of balls in $M$ are bounded by an exponential function of their radii. This implies that a Gaussian function is square-integrable.
\end{proof}

\section{$S_0^2$ and $S_1^2$ are $g$-trace class} \label{sec Sj2 g tr cl}

Let $S_0$ and $S_1$ be as in \eqref{eq def Sj}.
Our main goal in this section is to prove the following proposition.
\begin{proposition} \label{prop Sj2 g trace cl}
Under the conditions in Theorem \ref{thm taug indG}, the operators
 $S_0^2$ and $S_1^2$ are $g$-trace class.
\end{proposition}
In \cite{HWW}, it is shown that $S_0$ and $S_1$ are $g$-trace class in a general setting. An important subtlety is that this is true for the notion of $g$-trace class operators in Definition \ref{def g trace class}, which is relatively weak. For example, it does not reduce to the usual notion of trace class operators if $G$ is trivial, and it is not preserved by composition with bounded, or even other $g$-trace class operators. For this reason, Proposition \ref{prop Sj2 g trace cl} does not follow directly from the fact that $S_0$ and $S_1$ are $g$-trace class, and the arguments in this section are needed to prove it.

\subsection{Convergence of an integral for small $t$}

In this subsection and the next, we consider a general setting, where $E \to M$ is an equivariant, Hermitian vector bundle over a complete Riemannian metric with a  proper, isometric action by $G$. 

Let $D$ be a Dirac operator on $E$, assuming a Clifford action is given. Let $t_1 > 0$. In this subsection and the next, we suppose that the kernels of $e^{-tD^2}$ and $e^{-tD^2}D$ satisfy bounds of the type \eqref{eq heat kernel decay}, for $t \in (0,t_1]$.


We will use some calculus.
\begin{lemma} \label{lem int t0 t1}
Let $a,b>0$, and $t_0 \in (0, b/a]$. Then 
\beq{eq int t0 t1}
\int_0^{t_1} t^{-a}e^{-b/t}\, dt \leq t_1\min(t_0, t_1)^{-a} e^{-b/t_1}.
\eeq
\end{lemma}
\begin{proof}
The function $t\mapsto  t^{-a}e^{-b/t}$ is increasing on $(0,b/a]$, hence on $(0,t_0]$. So 
\beq{eq int 0 t0}
\int_0^{t_0} t^{-a}e^{-b/t}\, dt \leq t_0 t_0^{-a}e^{-b/t_0} \leq  t_0^{1-a}e^{-b/{t_1}},
\eeq
and a similar estimate holds for the integral from $0$ to $t_1$ if $t_1 \leq t_0$.
If $t_1 \geq t_0$, then
\beq{eq int t0 t1 2}
\int_{t_0}^{t_1} t^{-a}e^{-b/s}\, ds \leq 
 (t_1-t_0) t_0^{-a} e^{-b/t_1}. 
\eeq
The claim \eqref{eq int t0 t1} follows from a combination of \eqref{eq int 0 t0} and \eqref{eq int t0 t1 2}.
\end{proof}

\begin{lemma} \label{lem int small t}
Let $\kappa_t$ be the Schwartz kernel of either $e^{-tD^2}$ or $e^{-tD^2}D$. Let  $\varphi, \psi \in C^{\infty}(M)^{\Gamma}$ have supports separated by a positive distance $\varepsilon$, and let $\tilde \varphi, \tilde \psi \in C^{\infty}_c(M)$. Then the integral
\beq{eq int small t}
\sum_{\gamma \in \Gamma}
\left(
\int_{M \times M}  \tilde \varphi(m)\varphi(m) \tilde\psi(m')\psi(m')\int_{0}^{t_1}
 \| \gamma \kappa_t(\gamma^{-1}m, m')  \|^2\, dt \,dm\, dm'\right)^{1/2} 
\eeq
converges. 
\end{lemma}
\begin{proof}
For $\gamma \in \Gamma$, set 
\[
r(\gamma) := d(\gamma \supp (\tilde \varphi \varphi), \supp (\tilde \psi \psi)).
\]
The Gaussian bound \eqref{eq heat kernel decay} on $\kappa_t$ implies that for all $\gamma \in \Gamma$ and $t \in (0,t_1]$,
\begin{multline*}
\int_{M \times M} \tilde \varphi(m)\varphi(m) \tilde \psi(m')\psi(m')
 \| \gamma \kappa_t(\gamma^{-1}m, m')  \|^2 \,dm\, dm' \\
 = 
 \int_{M \times M} \tilde \varphi(\gamma m) \varphi(\gamma m) \tilde \psi(m')\psi(m')
 \|  \kappa_t(m, m')  \|^2 \,dm\, dm' \\
 \leq
 b_1^2 t^{-2b_2}e^{-2b_3 r(\gamma)^2/t} \|\tilde \varphi \varphi\|_{L^1} \|\tilde \psi \psi \|_{L^1}.
\end{multline*}
So
\begin{multline*}
\left(
\int_{M \times M} 
\tilde \varphi(m)\varphi(m) \tilde \psi(m')\psi(m')
\int_{0}^{t_1}
 \| \gamma \kappa_t(\gamma^{-1}m, m')  \|^2\, dt \,dm\, dm'\right)^{1/2} \\
 \leq
b_1  \|\tilde \varphi \varphi\|_{L^1}^{1/2} \|\tilde \psi \psi \|_{L^1}^{1/2}
  \left(
\int_{0}^{t_1} t^{-2b_2}
 e^{-2b_3 r(\gamma)^2/t}
 \right)^{1/2}. 
\end{multline*}

The assumptions on $\varphi$ and $\psi$ imply that $r(\gamma) \geq \varepsilon$ for all $\gamma \in \Gamma$. Set $t_0 := b_3 \varepsilon^2/b_2$. Then by Lemma \ref{lem int t0 t1},
\[
  \left( \int_{0}^{t_1} t^{-2b_2}
 e^{-2b_3 r(\gamma)^2/t} \right)^{1/2} \leq t_1^{1/2} \min(t_0, t_1)^{-b_2}e^{-b_3 r(\gamma)^2/{t_1}}.
\]

The Svarc--Milnor lemma and compactness of the supports of $\tilde \varphi$ and $\tilde \psi$ imply that there are $a,b>0$ such that for all $\gamma \in \Gamma$, $r(\gamma) \geq al(\gamma) - b$, where $l$ denotes the word length with respect to a fixed, finite, symmetric, generating set.
So there are $\alpha, \beta > 0$ such that for all $\gamma \in \Gamma$,
\[
e^{-b_3 r(\gamma)^2/{t_1}} \leq 
e^{-b_3(al(\gamma) - b)^2/{t_1}}
\leq \alpha e^{-\beta l(\gamma)^2/t_1}.
\]
The sum of the right hand side over $\gamma \in \Gamma$ converges, because of \eqref{eq growth Gamma}.
%
\end{proof}

\subsection{Convergence of an integral for large $t$}

We still consider a Dirac operator $D$, and now assume that $D^2 \geq c>0$.

As before, let $l$ be a word length function on $\Gamma$ with respect to a fixed, finite, symmetric, generating set. Because $\Gamma$ is finitely generated, there are $C,k>0$ such that \eqref{eq growth Gamma} holds  for all $n \in \N$.
Let  $\varphi, \psi \in C^{\infty}_c(M)$, and fix $m_0 \in \supp(\psi)$. 
Let $a_1$ and $a_2$ be as in \eqref{eq SM}.


%
 

\begin{proposition}\label{prop int QC}
Suppose that $M$ has bounded geometry.
Suppose that \eqref{eq growth Gamma} holds for a $k<\frac{2a_1 \sqrt{c}}{3}$.
Then for all $t_1>0$, the expression
\beq{eq int QC}
\sum_{\gamma \in \Gamma}  
\left(
\int_{M \times M}\int_{t_1}^{\infty}
 \varphi(m)\psi(m') \|\gamma e^{-s  D^2} D(\gamma^{-1}m, m')\|^2\, ds\, dm\, dm'\right)^{1/2} 
\eeq
converges.
\end{proposition}

By Lemma \ref{lem HS}, the operators $e^{-tD^2} \varphi$ and $e^{-tD^2}D \varphi$ 
 are Hilbert--Schmidt for all $t>0$.
\begin{lemma} \label{lem HS decay}
For all $\varphi \in C^{\infty}_c(M)$, and all $t_1>0$, there exists an $a>0$ such that for all $t>t_1$,
\[
\begin{split}
\|e^{-tD^2}\varphi\|_{\HS} &\leq ae^{-ct};\\
\|e^{-tD^2}D\varphi\|_{\HS} &\leq ae^{-ct}.\\
\end{split}
\]
\end{lemma}
\begin{proof}
For $t>0$, let  $A_t$ be either the operator $e^{-tD^2}$ or $e^{-tD^2}D$. 
Then for all $t > t_1 > 0$, and all $s \in L^2(E)$,
\[
\begin{split}
\|A_{t} \varphi s\|^2 &= \|e^{-(t - t_1)D^2}A_{t_1} \varphi s\|^2 \\
&= \bigl(e^{-2(t - t_1)D^2}A_{t_1} \varphi s, A_{t_1} \varphi s \bigr) \\
& \leq e^{-2c(t - t_1)} \|A_{t_1} \varphi s\|^2.
\end{split}
\]
Let $\{e_j\}_{j=1}^{\infty}$ be an orthonormal basis of $L^2(E)$. Then by the above estimate,
\[
\|A_{t} \varphi \|_{\HS}^2 = \sum_{j=1}^{\infty} \|A_{t}\varphi e_j\|^2 \leq e^{-2c(t - t_1)} \|A_{t_1} \varphi \|_{\HS}^2.
\]
\end{proof}

Let $\varphi, \psi \in C^{\infty}_c(M)$, and suppose
 for simplicity that these functions take values in $[0,1]$. 
%
For $\gamma \in \Gamma$, set
\[
r(\gamma) := d(\gamma \supp(\varphi),  \supp(\psi)).
\]
(Here we note that $r(\gamma)$ may be zero.)
Fix $\gamma \in \Gamma$ and $t>0$. Let $\zeta \in C^{\infty}_c(M)$ be a function with values in $[0,1]$ such that
\beq{eq nu}
\begin{split}
d(\supp(\psi), \supp(1-\zeta)) &\geq r(\gamma)/3;\\
d(\gamma \supp(\varphi), \zeta) &\geq r(\gamma)/3.\\
\end{split}
\eeq
Write 
\[
(\gamma \cdot \varphi)e^{-t D^2}D \psi = A(\gamma) + B(\gamma),
\]
where
\[
\begin{split}
A(\gamma) &:= (\gamma \cdot \varphi)e^{-t D^2/2} \zeta e^{-t D^2/2} D \psi;\\
B(\gamma) &:= (\gamma \cdot \varphi)e^{-t D^2/2} (1-\zeta) e^{-t D^2/2} D \psi.
\end{split}
\]

\begin{lemma} \label{lem A gamma}
The operator $A(\gamma)$ is Hilbert--Schmidt, and there is a  $b> 0$, independent of $\gamma$, such that for all $t\geq t_1$,
\[
\|A(\gamma)\|_{\HS} \leq b e^{-r(\gamma)^2/9t - ct}.
\]
\end{lemma}
\begin{proof}
For all $s \in L^2(E)$ and $\gamma \in \Gamma$,
\[
\begin{split}
\|A(\gamma)s\| & \leq \|(\gamma \cdot \varphi)e^{-t D^2/2} \zeta^{1/2} \|_{\cB(L^2(E))} 
\| \zeta^{1/2} e^{-t D^2/2} D \psi s\|.
\end{split}
\]
By Corollary \ref{cor hk decay} and the second inequality in \eqref{eq nu},
\[
\|(\gamma \cdot \varphi)e^{-t D^2/2} \zeta^{1/2} \|_{\cB(L^2(E))}  \leq 
 \frac{2}{ \sqrt{\pi}} e^{-\frac{r(\gamma)^2}{18t}}.
\]
So, if $\{e_j\}_{j=1}^{\infty}$ is an orthonormal basis of $L^2(E)$,
\[
\|A(\gamma)\|_{\HS}^2 \leq 
 \frac{4}{\pi} e^{-\frac{r(\gamma)^2}{9t}}
\sum_{j=1}^{\infty} \| \zeta^{1/2} e^{-t D^2/2} D \psi e_j\|^2 \leq 
 \frac{4}{\pi} e^{-\frac{r(\gamma)^2}{9t}}
 \|e^{-t D^2/2} D \psi \|_{\HS}^2.
\]
The claim now follows by Lemma \ref{lem HS decay}.
\end{proof}

\begin{lemma}\label{lem B gamma}
The operator $B(\gamma)$ is Hilbert--Schmidt, and  there is a  $b> 0$, independent of $\gamma$, such that for all $t\geq t_1$,
\[
\|B(\gamma)\|_{\HS} \leq b e^{-r(\gamma)^2/9t - ct}.
\]
\end{lemma}
\begin{proof}
The operator $B(\gamma)$ is Hilbert--Schmidt if and only its adjoint is, and then these operators have the same Hilbert--Schmidt norm. Now
\[
\begin{split}
B(\gamma)^* &= \psi e^{-t D^2/2} D (1-\zeta) e^{-t D^2/2}   (\gamma \cdot \varphi) \\
&=  \psi e^{-t D^2/2}  (1-\zeta) D e^{-t D^2/2}   (\gamma \cdot \varphi) -  \varphi e^{-t D^2/2} c(d\zeta) e^{-t D^2/2}   (\gamma \cdot \varphi).
\end{split}
\]
The distance between the supports of $\varphi$ and $1-\zeta$ is at least $r(\gamma)/3$. The support of $d\zeta$ lies inside the support of $1-\zeta$, so the distance between the supports of $\varphi$ and $d\zeta$  is at least $r(\gamma)/3$ as well. So Corollary \ref{cor hk decay} implies that
\[
\begin{split}
\| \psi e^{-t D^2/2}  (1-\zeta) \| &\leq  \frac{2}{ \sqrt{\pi}} e^{-\frac{r(\gamma)^2}{18t}};\\
\| \psi e^{-t D^2/2}  c(d\zeta) \| &\leq  \|d\zeta\|_{\infty} \frac{2}{ \sqrt{\pi}} e^{-\frac{r(\gamma)^2}{18t}}.
\end{split}
\]


And the Hilbert--Schmidt norms of
\[
 e^{-t D^2/2}   (\gamma \cdot \varphi) = \gamma  e^{-t D^2/2}    \varphi \gamma^{-1}
\]
and
\[
D  e^{-t D^2/2}   (\gamma \cdot \varphi) = \gamma  D e^{-t D^2/2}    \varphi \gamma^{-1}
\]
are independent of $\gamma$.
 So a similar argument to the proof of Lemma \ref{lem A gamma} applies to show that
 there is a $b > 0$ such that for all $t\geq t_1$,
\[
\|B(\gamma)\|_{\HS} = 
\|B(\gamma)^*\|_{\HS} \leq b  e^{-r(\gamma)^2/9t - ct}.
\]
\end{proof}

\begin{lemma} \label{lem conv sum gamma}
Let  $C, k, \alpha_1, \alpha_2, \alpha_3, t_1> 0$, and suppose that \eqref{eq growth Gamma} holds for all $n \in \N$.
Suppose that $k^2 < 4\alpha_1 \alpha_3$.
Then
\beq{eq conv 2}
\sum_{\gamma \in \Gamma}  \int_{t_1}^{\infty} 
 e^{-\alpha_1  \frac{(l(\gamma) - \alpha_2)^2}{s} - \alpha_3 s}\, ds
\eeq
converges.
\end{lemma}
\begin{proof}
The sum 
 \eqref{eq conv 2} equals
\begin{multline}\label{eq conv 1}
 \sum_{n=0}^{\infty} \sum_{\gamma \in \Gamma; l(\gamma)=n}  \int_{t_1}^{\infty} 
  e^{-\alpha_1  \frac{(n- \alpha_2)^2}{s} - \alpha_3 s}
 \, ds 
 \leq
  C \sum_{n=0}^{\infty}    \int_{t_1}^{\infty} 
  e^{-\alpha_1  \frac{(n- \alpha_2)^2}{s} - \alpha_3 s + kn}
 \, ds 
 \\
 = C e^{k  \alpha_2}  \int_{t_1}^{\infty} 
e^{\left(\frac{k^2}{4 \alpha_1 } - \alpha_3\right)s} 
 \left(
 \sum_{n=0}^{\infty} 
 e^{ -\frac{\alpha_1}{s}\left(  n - \alpha_2 - \frac{ks}{2\alpha_1}\right)^2}
 \right)
 \, ds.
 \end{multline}
(Because all terms and integrands are positive, convergence does not depend on the order of summation and integration.)
%
Convergence of the right hand side of \eqref{eq conv 1} is equivalent to convergence of the double integral
\beq{eq conv 3}
 \int_{t_1}^{\infty} 
e^{\left(\frac{k^2}{4 \alpha_1 } - \alpha_3\right)s} 
 \left(
\int_{0}^{\infty} 
 e^{ -\frac{\alpha_1}{s}\left(  x - \alpha_2 - \frac{ks}{2\alpha_1}\right)^2}\, dx
 \right)
 \, ds.
%
%
\eeq
And for all $s >0$,
\[
\int_{0}^{\infty} 
 e^{ -\frac{\alpha_1}{s}\left(  x - \alpha_2 - \frac{ks}{2\alpha_1}\right)^2}\, dx
 \leq 
 \int_{\R}
 e^{ -\frac{\alpha_1}{s}\left(  x - \alpha_2 - \frac{ks}{2\alpha_1}\right)^2}\, dx
= \sqrt{\frac{\pi s}{\alpha_1}}.
\]
We find that a sufficient condition for the convergence of \eqref{eq conv 3} is convergence of 
\[
 \int_{t_1}^{\infty} 
e^{\left(\frac{k^2}{4 \alpha_1 } - \alpha_3\right)s} 
\sqrt{\frac{\pi s}{\alpha_1}}
 \, ds.
%
\]
This is equivalent to the condition $k^2 < 4 \alpha_1\alpha_3$.
%
%
%
\end{proof}

\begin{proof}[Proof of Proposition \ref{prop int QC}]
The integral \eqref{eq int QC} equals
\beq{eq ind HS QC}
\sum_{\gamma \in \Gamma} \left\| \int_{t_1}^{\infty} 
\varphi \circ \gamma \circ  e^{-s  D^2} D \circ \psi \, ds
\right\|_{\HS} \leq 
\sum_{\gamma \in \Gamma}  \int_{t_1}^{\infty}
 \| \varphi \circ \gamma \circ  e^{-s  D^2} D \circ \psi \|_{\HS}\, 
 ds.
\eeq
By Lemmas \ref{lem A gamma} and \ref{lem B gamma}, there is a  $b > 0$ such that for all $t\geq t_1$ and all $\gamma \in \Gamma$,
\[
 \| \varphi \circ \gamma \circ  e^{-s  D^2} D \circ \psi \|_{\HS} = 
  \| (\gamma \cdot \varphi) \circ   e^{-s  D^2} D \circ \psi \|_{\HS} 
  \leq b e^{ -r(\gamma)^2/9s - cs}.
\]

The condition \eqref{eq SM} and compactness of $\supp(\varphi)$ and $\supp(\psi)$ imply that there is $a_3>0$ such that for all $\gamma \in \Gamma$, $r(\gamma) \geq a_1 l(\gamma) - a_3$.
So
\[
 \| \varphi \circ \gamma \circ  e^{-s  D^2} D \circ \psi \|_{\HS} 
  \leq 
  b e^{-  \frac{(a_1l(\gamma)-a_3)^2}{9s} - cs}. 
\]
 So the right hand side of \eqref{eq ind HS QC} is at most equal to
\[
b
\sum_{\gamma \in \Gamma}  \int_{t_1}^{\infty} 
e^{-  \frac{(a_1l(\gamma)-a_3)^2}{9s} - cs}
 \, ds. 
\]
By Lemma \ref{lem conv sum gamma}, this converges if $\Gamma$ satisfies \eqref{eq growth Gamma} for some $C, k>0$ with 
$
k^2 <  \frac{4 a_1^2c}{9}
$
.
%
%
\end{proof}

\subsection{Proof of Proposition \ref{prop Sj2 g trace cl}}\label{sec pf Sj2 g trace cl}

We return to the setting of Subsection \ref{sec bdry}, where $M$ is a manifold with boundary $N$, on which $G$ acts cocompactly, and $\hat M$ is obtained from $M$ by attaching a cylinder $N \times [0,\infty)$.

The operators $\tilde Q$ and $Q_C$ as in Subsection \ref{sec Sj}  
do not have smooth kernels, but if $\varphi, \psi \in C^{\infty}(M)$ have disjoint supports, then $\varphi \tilde Q \psi$ and $\varphi Q_C \psi$ do.
\begin{lemma} \label{lem tilde Q summable}
Suppose that the kernel of the operator $\tilde De^{-t \tilde D^2}$ satisfies a bound of the type \eqref{eq heat kernel decay}.
If $\varphi, \psi \in C^{\infty}(M)^{\Gamma}$ have  supports separated by a postive distance, then 
$\varphi \tilde Q \psi$ is $\Gamma$-summable.
\end{lemma}
\begin{proof}
We have
\[
\tilde Q = -\int_0^t  e^{-s \tilde D_+ \tilde D_-} \tilde D_- \, ds.
\]
So the claim follows from Lemma \ref{lem int small t}.
\end{proof}

\begin{proposition} \label{prop QC summable}Consider the setting of 
Theorem \ref{thm taug indG}(b).
If $\varphi, \psi \in C^{\infty}(M)^{\Gamma}$ have  supports separated by a postive distance, then  $\varphi Q_C \psi$ is $\Gamma$-summable.
\end{proposition}
\begin{proof}
We have
\[
Q_C = \int_0^{\infty}  e^{-s  (D_C)_+  (D_C)_-}  (D_C)_- \, ds.
\]
The operator
\[
\varphi  \int_0^{1}  e^{-s  (D_C)_+  (D_C)_-}  (D_C)_- \, ds\,  \psi
\]
is $\Gamma$-summable by Lemma \ref{lem int small t}, and the operator
\[
\varphi  \int_1^{\infty}  e^{-s  (D_C)_+  (D_C)_-}  (D_C)_- \, ds\,  \psi
\]
is $\Gamma$-summable by Proposition \ref{prop int QC}. The coefficient that appears $a_1$ in \eqref{eq SM} and in the growth condition on $\Gamma$ is independent of the choice of $m_0 \in \supp(\psi)$ by compactness of $M/\Gamma$ and $\Gamma$-invariance of the distance on $M$.  
\end{proof}

Let the functions $\varphi_j$ and $\psi_j$, and the operator $S_0$ 
be as in Subsection \ref{sec Sj}, and let
and $\tilde S_0$ be as in \eqref{eq tilde Sj}.
\begin{proposition} \label{prop S_0 summable}
Consider the setting of Theorem \ref{thm taug indG}(b).
The operator $(\varphi_1 \tilde Q - \varphi_2 Q_C) \psi_1'$ has a smooth kernel, and is $\Gamma$-summable.
\end{proposition}
\begin{proof}
The operator $S_0$ has a smooth kernel by Lemma \ref{lem Sj smoothing}, and $\varphi_1 \tilde S_0 \psi_1$ has a smooth kernel as well. Hence so does 
\[
(\varphi_1 \tilde Q - \varphi_2 Q_C) \psi_1' = S_0 - \varphi_1 \tilde S_0 \psi_1.
\]

As in the proof of Proposition 5.7 in \cite{HWW}, 
\begin{multline}\label{eq S_0 summable}
(\varphi_1 \tilde Q - \varphi_2 Q_C) \psi_1' = (\varphi_1 \tilde Q - \varphi_2 Q_C') \psi_1' - 
\varphi_2(Q_C - Q_C')\psi_1'\\
 =
(\varphi_1 \tilde Q - \varphi_2 Q_C') \psi_1' -
\varphi_2 e^{-D_{C, +} D_{C, -}}Q_C \psi_1' \\
-\int_0^t \bigl( \varphi_1e^{-s \tilde D_+ \tilde D_-}  -  \varphi_2 e^{-sD_{C, +} D_{C, -}}\bigr)D_-\, ds\, \sigma\psi_1' 
- \varphi_2 \int_t^{\infty} 
e^{-s D_{C, -}D_{C, +}}D_{C, -} \, ds\,  \sigma \psi_1'
\end{multline}
The second term on the right hand side is $\Gamma$-summable by Proposition \ref{prop int QC}, in which it is not assumed that the functions $\varphi$ and $\psi$ have disjoint supports. Here we again use the fact that the coefficient $a_1$ that appears  in \eqref{eq SM} and in the growth condition on $\Gamma$ is independent of the choice of $m_0 \in \supp(\psi)$ by compactness of $M/\Gamma$ and $\Gamma$-invariance of the distance on $M$.


We now focus on the first term on the right hand side of \eqref{eq S_0 summable}.
As in the proof of Lemma 5.5 in \cite{HWW}, let $\varphi \in C^{\infty}(\hat M)$ be such that for $j=1,2$, $\varphi$ equals $1$ on the support of $\psi_j'$, and zero outside the support of $1-\varphi_j$.
Since $(1-\varphi)$ and $\psi_1'$ have supports separated by a positive distance, Lemma \ref{lem int small t} implies that
\[
\int_0^t (1-\varphi) \bigl( \varphi_1e^{-s \tilde D_+ \tilde D_-}  -  \varphi_2 e^{-s(D_{C, +} D_{C, -})}\bigr)D_-\sigma\psi_1'\, ds
\]
is $\Gamma$-summable.

Let $\tilde \varphi, \tilde \psi \in C^{\infty}_c(M)$. 
Then as in Lemma 5.4 in \cite{HWW}, for all $m,m' \in M$
\[
\bigl(\tilde \varphi \varphi \bigl( \varphi_1e^{-s \tilde D_+ \tilde D_-}  -  \varphi_2 e^{-s(D_{C, +} D_{C, -})}\bigr)D_-\sigma\psi_1' \tilde \psi\bigr)(m,m') = \frac{1}{(2\pi s)^{\dim(M)/2}}e^{-d(m,m')^2/4s} F(s, m, m'),
\]
where $F(s, m, m')$
vanishes to all orders in $s$ as $s \downarrow 0$, uniformly in $m,m'$ in compact sets. This implies that $\tilde  \varphi \varphi \bigl( \varphi_1e^{-s \tilde D_+ \tilde D_-}  -  \varphi_2 e^{-s(D_{C, +} D_{C, -})}\bigr)D_-\sigma\psi_1'$ is $\Gamma$-summable via a simpler version of the proof of Lemma \ref{lem int small t}. 
\end{proof}

\begin{proof}[Proof of Proposition \ref{prop Sj2 g trace cl}]
First suppose that $G/Z_g$ is compact. Because the functions $\varphi_1$ and $\varphi_2'$ are cocompactly supported, Lemma \ref{lem comp Sj} implies that there is a cocompactly supported function $\varphi \in C^{\infty}(M)^G$ such that $\varphi S_1 = S_1$. So $S_1^2$ is $g$-trace class by Lemma \ref{lem GZ cpt} and the comment below it. And because $\psi_1$ and $\psi_2'$ are cocompactly supported, Lemma \ref{lem comp Sj}  implies that there is a cocompactly supported function $\varphi \in C^{\infty}(M)^G$ such that $S_0 \varphi  = S_0$. So $S_0^2$ is $g$-trace class, again by Lemma \ref{lem GZ cpt}. Part (a) follows.

For case (b) in Theorem \ref{thm taug indG}, suppose that $G = \Gamma$ is discrete. The operator $\tilde S_1$ is $\Gamma$-summable, so Lemma \ref{lem tilde Q summable} and Proposition \ref{prop QC summable} imply that the three terms in the expression for $S_1$ in Lemma \ref{lem comp Sj}  are all $\Gamma$-summable. As in the proof of part (a),  there is a cocompactly supported function $\varphi \in C^{\infty}(M)^G$ such that $\varphi S_1 = S_1$. So $S_1^2$ is $g$-trace class by Lemma \ref{lem gamma tr cl}. 

The operator $\tilde S_0$ is $\Gamma$-summable, so Proposition \ref{prop S_0 summable} and 
Lemma \ref{lem comp Sj} 
 imply that
\[
S_0 = \varphi_1 \tilde S_0 \psi_1 + (\varphi_1 \tilde Q - \varphi_2 Q_C) \psi_1'
\]
is $\Gamma$-summable as well. As in the proof of part (a),  there is a cocompactly supported function $\varphi \in C^{\infty}(M)^G$ such that $S_0 \varphi = S_0$. So $S_0^2$ is $g$-trace class by Lemma \ref{lem gamma tr cl}. 
\end{proof}

\section{The trace of the index}\label{sec trace index}

Our main goal in this section is to prove the following part of Theorem \ref{thm taug indG}.
\begin{proposition} \label{prop coarse index}
If $S_0^2$ and $S_1^2$ are $g$-trace class, then
\beq{eq g trace coarse index}
\tau_g(\ind_G(\hat D)) = \Tr_g(S_0^2) - \Tr_g(S_1^2).
\eeq
\end{proposition}
Together with Proposition \ref{prop Sj2 g trace cl}, this is the main part of the proof of Theorem \ref{thm taug indG}.

\subsection{An explicit index}

Let $C^{\infty}(\hat M; L^2(\hat E))^G_{\loc}$ be the subalgebra of elements of $C^*(\hat M; L^2(\hat E))^G_{\loc}$ with smooth kernels.
Because $\hat D$ is a multiplier of $C^{\infty}(\hat M; L^2(\hat E))^G_{\loc}$,
Lemmas \ref{lem Sj smoothing} and \ref{lem S0 S1} imply that
\beq{eq def e}
e:=
\left(
\begin{array}{cc}
S_0^2 & S_0(1+S_0)R \\
S_1 \hat D_+ & 1-S_1^2
\end{array}
 \right)
\eeq
is an idempotent in $C^{\infty}(\hat M; L^2(\hat E))^G_{\loc}$. (The $2\times 2$ matrix notation is with respect to the decomposition $\hat E =\hat E_+ \oplus \hat E_-$.) See also page 353 of \cite{CM90}. We write
\[
p_2 := \left( \begin{array}{cc} 0 & 0 \\ 0 & 1\end{array}\right).
\]


Let
\[
\iota\colon C^{\infty}(\hat M; L^2(\hat E))^G_{\loc} \to C^*(\hat M; L^2(\hat E))^G_{\loc}
\]
be the inclusion map.
Let
\[
\ind_G^{L^2(E)}(\hat D) \in K_0(C^{*}(\hat M; L^2(\hat E))^G_{\loc})
\]
be defined as in \eqref{eq def L2E ind}.
\begin{lemma} \label{lem coarse idempotent}
We have
\beq{eq coarse index}
\ind_G^{L^2(\hat E)}(\hat D) = \iota_*([e] - [p_2]).
\eeq
\end{lemma}
\begin{proof}
The right hand side of \eqref{eq coarse index} equals
$
\partial[\hat D],
$
where
\[
\partial\colon K_1(\cM(C^{\infty}(\hat M; L^2(\hat E)^G_{\loc})/C^{\infty}(\hat M; L^2(\hat E)^G_{\loc}) \to K_0(C^{\infty}(\hat M; L^2(\hat E)^G_{\loc})
\]
is the boundary map in the six-term exact sequence. The image of $\partial[\hat D]$ in $K_0(C^{*}(\hat M;  L^2(\hat E)^G)_{\loc})$
 equals
$
 [\bar e] - [p_2],
$
where $\bar e$ is the idempotent defined as the right hand side of \eqref{eq def e}, with $R$ replaced by $\bar R$, and $S_j$ by $\bar S_j$, for any multiplier $\bar R$ of $C^{\infty}(\hat M; L^2(\hat E)^G_{\loc}$ such that $\bar S_0 := 1-\bar R \hat D_+$ and $\bar S_1 := 1- \hat D_+ \bar R$ are in $C^{\infty}(\hat M; L^2(\hat E))^G_{\loc}$. In other words, for any such $\bar R$,
\beq{eq e bar e}
[e]-[p_2] = [\bar e] - [p_2].
\eeq

Let $b$ de the function used in Subsection \ref{sec def index}. We now choose $b$ such that $b(x) = O(x)$ as $x \to 0$, so that the function $x\mapsto b(x)/x$ has a continuous extension to $\R$.
The function $b$
 is odd, and the function $x\mapsto b(x)/x$ is even. So the operator
$\frac{b(\hat D)}{\hat D}$ is even with respect to the grading on $E$, whereas $b(\hat D)$ is odd. We
denote restrictions of operators to sections of $E_{\pm}$ be subscripts $\pm$, respectively.
We choose
\[
\bar R := b(\hat D)_- \Bigl(\frac{b(\hat D)}{\hat D}\Bigr)_-.
\]
Then
we obtain operators $\bar S_0$ and $\bar S_1$ which equal the restrictions of $1-b(\hat D)^2$ to even and odd graded sections of $E$, respectively. We claim that $\bar S_0$ and $\bar S_1$ lie in $C^{\infty}(\hat M; L^2(\hat E)^G_{\loc}$. Indeed, 
by Lemma 2.3 in \cite{Roe16}, these operators lie in $C^{*}(\hat M;  L^2(\hat E))^G_{\loc}$. 
And $1-b^2$ is compactly supported, so $\hat D^j (1-b(\hat D)^2)$ is a bounded operator on $L^2(\hat E)$ for all $j \in \N$. Hence, by elliptic regularity, $1-b(\hat D)^2$ maps $L^2(\hat E)$, and any Sobolev space defined in terms of $\hat D$, continuously into $\Gamma^{\infty}(E)$. So this operator has a smooth kernel.


 For this choice of $\bar R$, we have
\[
\bar e = \left(
\begin{array}{cc}
\bar S_0^2 & \bar S_0(1+ \bar S_0)b(\hat D)_- \Bigl(\frac{b(\hat D)}{\hat D}\Bigr)_-\\
\bar S_1 \hat D_+ & 1- \bar S_1^2
\end{array}
 \right)
 \]
For $s \in [0,1]$, we write
 \[
 A_s :=
 \left(
\begin{array}{cc}
\Bigl(\frac{b(\hat D)}{\hat D}\Bigr)^{-s/2}_+& 0 \\
0 & \Bigl(\frac{b(\hat D)}{\hat D}\Bigr)^{s/2}_-
\end{array}
 \right),
\]
and consider the idempotent
\[
e_s :=   A_s \bar e A_s^{-1} = \left(
\begin{array}{cc}
\bar S_0^2 & \bar S_0(1+ \bar S_0) b(\hat D)_- \left(\frac{b(\hat D)}{\hat D} \right)^{1-s}_-
  \\
\bar S_1 \hat D_+\left(\frac{b(\hat D)}{\hat D} \right)^{s}_+ & 1- \bar S_1^2
\end{array}
 \right)
\]
in $M_2(C^*(\hat M; L^2(\hat E)^G_{\loc}$.
Via this continuous path of idempotents, we conclude from \eqref{eq e bar e} that
\[
[e] - [p_2] = [\bar e] - [p_2] = 
 [e_1] - [p_2].
\]
By the definition \eqref{eq def L2E ind} of $\ind_G^{L^2(\hat E)}(\hat D)$, this index equals $[e_1] - [p_2]$. The map $\iota_*$ may be inserted here because the entries of $e_1$ have smooth kernels.
\end{proof}

\subsection{The map $\tilTR$}\label{sec tilde TR}

In this subsection, we temporarily return to the general setting of Subsection~\ref{sec Roe alg}.
Because $Z/G$ is compact, the  equivariant Roe algebra $C^*(Z; L^2(E|_Z))^G$ equals the closure in $\cB(L^2(E|_Z))$ of the algebra of bounded operators on $L^2(E|_Z)$ with finite propagation, and $G$-invariant, continuous kernels
\beq{eq kappa X}
\kappa \in \Gamma(Z \times Z, \End(E|_Z)).
\eeq
This can be proved analogously to the arguments in Section 5.4 in \cite{GHM}. We will not need this fact, however, since the operators in $C^*(Z; L^2(E|_Z))^G_{\loc}$ we work with always have continuous kernels.

Let $\chi \in C(X)$ be a cutoff function for the action by $G$, as in \eqref{eq def cutoff}.
Define the map
\[
\tilTR\colon C^*(Z; L^2(E|_Z))^G \to C^*_rG \otimes \cK(L^2(E|_Z))
\]
by
\[
\tilTR(\kappa)(h) = T_{(\chi \otimes \chi) h\cdot \kappa},
\]
for $h \in G$ and $\kappa$ as in \eqref{eq kappa X}. Here $T_{(\chi \otimes \chi) h\cdot \kappa}$ is the operator whose Schwartz kernel is given by
\[
((\chi \otimes \chi) h\cdot \kappa)(z,z') = {\chi}(z)  \chi(z') h\kappa(h^{-1}z,z'),
\]
for all $h \in G$ and $z, z' \in Z$.
(The map $\tilTR$ is not a trace, the notation is motivated by Lemma \ref{lem TR Kthry} below.)
\begin{lemma} \label{lem tilde TR}
The map $\tilTR$ is an injective $*$-homomorphism.
\end{lemma}
\begin{proof}
The fact that  $\tilTR$ is a $*$-homomorphism follows from direct computations involving $G$-invariance of $\kappa$. It follows from $G$-invariance of $\kappa$ that $\kappa = 0$ if $(\chi \otimes \chi) h\cdot \kappa = 0$ for all $h \in G$.
\end{proof}

Let $C^*_{\Tr}(Z; L^2(E|_Z))^G \subset C^*(Z; L^2(E|_Z))^G$ be the subalgebra of operators with kernels $\kappa$ such that $\widetilde{\TR}(\kappa) \in C^*_rG \otimes \mathcal{L}^1(L^2(E|_Z))$, where $\mathcal{L}^1$ stands for the space of trace-class operators.

Analogously to Subsection 3.4 of \cite{HWW}, we define
\[
\TR(\kappa)(x) := \int_Z\chi(xm)^2 \tr(x\kappa(x^{-1}m,m))\, dm,
\]
for $\kappa \in \Gamma^{\infty}(\End(E|_Z))^G$ and $x \in G$ for which the integral converges.
\begin{lemma} \label{lem TR Kthry}
For all $\kappa \in C^*_{\Tr}(Z; L^2(E|_Z))^G$ and $x \in G$,
\[
\TR(\kappa)(x) = \Tr(\widetilde{\TR}(\kappa)(x)).
\]
\end{lemma}
\begin{proof}
For any $G$-equivariant operator $T$ on $L^2(E|_Z)$ with smooth kernel $\kappa \in C^*_{\Tr}(Z; L^2(E|_Z))^G$,
 and any $x \in G$, the trace property of the operator trace $\Tr$ and $G$-equivariance of $T$ imply that
\[
\TR(T)(x)=
\Tr(x \chi^2 T) = \Tr(\chi xT \chi) = \Tr(\widetilde{\TR}(\kappa)(x)).
\]
\end{proof}

\begin{lemma} \label{lem taug TR Trg}
For all $\kappa \in C^*_{\Tr}(Z; L^2(E|_Z))^G$ such that $ \Tr \circ \tilTR(\kappa) \in \cA$, 
\[
\tau_g \circ \Tr \circ \tilTR(\kappa) = \Tr_g(\kappa).
\]
\end{lemma}
\begin{proof}
It is immediate from the definitions that $\Tr_g = \tau_g \circ \TR$. So the claim 
follows from Lemma \ref{lem TR Kthry}.
\end{proof}

\subsection{Two maps from Roe algebras to $C^*_rG \otimes \cK$}

To apply $\tau_g$ to the localised coarse index of an operator, one needs a specific isomorphism \eqref{eq iso Roe group}. The key step in the proof of Proposition \ref{prop coarse index} is the fact that two maps from localised Roe algebras to group $C^*$-algebras tensored with the algebra of compact operators lead to the same result when one applies $\tau_g$. See Proposition \ref{prop diagram indices}. One of these maps  is the one applied in \cite{GHM} to map the localised equivariant coarse index into the $K$-theory of a group $C^*$-algebra. The other is defined in terms of the map $\tilTR$ from Subsection \ref{sec tilde TR}, and is suitable for computing $g$-traces.

Let $X$ be a proper, isometric, Riemannian $G$-manifold, and let $Z \subset X$ be a cocompact subset.
Suppose that $Z = G \times_K Y$ for a slice $Y \subset Z$ and a compact subgroup $K<G$. (We comment on how to remove this assumption in Remark \ref{rem L2Z L2E}.)
Fix a Borel section $\phi\colon K\backslash G\to G$.
The map
\beq{eq def psi}
\psi\colon Z \times G \to G \times K\backslash G \times Y
\eeq
given by
\[
\psi(gy, h) = \bigl(h\phi(Kg^{-1}h)^{-1}, Kg^{-1}h, \phi(Kg^{-1}h)h^{-1}gy\bigr)
\]
for $g,h \in G$ and $y \in Y$, is $G$-equivariant and bijective, with respect to the diagonal action by $G$ on $Z \times G$ and the action by $G$ on the factor $G$ on the right hand side of \eqref{eq def psi}. We always use the action by $G$ on itself by left multiplication. The map $\psi$  relates the measures $dz\,  dg$ and $dg \, d(Kg) \,dy$ to each other, as shown in Lemma 5.2 in \cite{GHM}.

Let $E \to X$ be a $G$-equivariant, Hermitian vector bundle.
Write
\[
H := L^2(K \backslash G) \otimes L^2(E|_Y).
\]
Then pulling back along $\psi$ defines a $G$-equivariant, unitary isomorphism
\beq{eq psi star}
\psi^*\colon L^2(G) \otimes H \to L^2(E|_Z) \otimes L^2(G).
\eeq

Let $\psi_1$ and $\psi_2$ be the projections of $\psi$ onto $G$ and $K \backslash G \times Y$, respectively.
Define the map
\[
\eta\colon Z \to K \backslash G \times Y
\]
by
\[
\eta(z) = \psi_2(z,e).
\]
This induces a unitary isomorphism
\[
\eta^*\colon H \to L^2(E|_Z).
\]


Let $C^*_{\ker}(Z)^G$ be the algebra as in Definition 5.10 in \cite{GHM}, of continuous kernels
\[
\kappa_G\colon G \times G \to \cK(H)
\]
with finite propagation, and the invariance property that for all $g,g',h \in G$,
\beq{eq kappa G invar}
\kappa_G(hg, hg') = \kappa_G(g, g').
\eeq
Such a kernel defines an operator on $L^2(G) \otimes H$, which corresponds to an operator on $L^2(E|_Z)\otimes L^2(G)$ via \eqref{eq psi star}. This gives a map
\[
a\colon C^*_{\ker}(Z)^G \to C^*(Z)^G
\]
with dense image; see Proposition 5.11 in \cite{GHM}. We also have an injective $*$-homomorphism
\[
W\colon C^*_{\ker}(Z)^G  \to C^*_rG \otimes \cK(H)
\]
with dense image, given by
\[
W(\kappa_G)(g) = \kappa_G(g^{-1}, e),
\]
for $\kappa_G \in C^*_{\ker}(Z)^G$ and $g \in G$.

There are natural maps
\beq{eq phi phiE}
\begin{split}
\varphi\colon C^*(Z)^G &\to C^*(X)^G_{\loc};\\
\varphi_E\colon C^*(Z; L^2(E|_Z))^G &\to C^*(X; L^2(E))^G_{\loc},
\end{split}
\eeq
defined by extending operators by zero outside $Z$, 
that induce isomorphisms on $K$-theory; see Section 7.2 in \cite{GHM}.
Consider the map $\oplus 0$ from \eqref{eq def plus 0}.
\begin{proposition} \label{prop diagram indices}
The diagram
\beq{eq diag indices}
\xymatrix{
C^*(X; L^2(E))^G_{\loc} \ar[r]^-{\oplus 0}& C^*(X)^G_{\loc}\\
 \ar[u]^-{\varphi_E} C^*(Z; L^2(E|_Z))^G
  \ar[dd]_-{\tilTR} & C^*(Z)^G \ar[u]_{\varphi} \\
  & C^*_{\ker}(Z)^G \ar[u]_a \ar[d]^-{W} \\
   C_c(G) \otimes \cK(L^2(E|_Z)) \ar[d]_-{\tau_g \otimes 1}
 &
 C_c(G) \otimes \cK(H)\ar[d]^-{\tau_g \otimes 1} \\
 \cK(L^2(E|_Z)) \ar[r]^-{\eta^*} & \cK(H).
}
\eeq
commutes in the following sense: the maps $a$, $\varphi_E$ and $\varphi$ are injective, with dense images, and the diagram commutes on the relevant dense subalgebras for the inverses of these maps. More explicitly, if $\kappa \in C^*(Z; L^2(E|_Z))^G$, $\kappa_G \in C^*_{\ker}(Z)^G$ and $\varphi_E(\kappa) \oplus 0 = \varphi \circ a (\kappa_G)$, then
\[
\eta^* \circ (\tau_g \otimes 1) \circ \tilTR(\kappa) = (\tau_g \otimes 1) \circ W(\kappa_G).
\]
\end{proposition}

\begin{remark}\label{rem L2Z L2E}
In general, $Z$ is a finite disjoint union of subsets of the form $Z_j = G\times_{K_j} Y_j$; see \cite{Palais61}. We can generalise Proposition \ref{prop diagram indices}  to that setting,
by viewing operators on $L^2(E|_Z)$ as finite matrices of operators between the spaces $L^2(E|_{Z_j})$, and comparing them with analogous matrices of operators between the spaces $H_j := L^2(K_j \backslash G) \otimes L^2(E|_{Y_j})$.
%
 \end{remark}

\subsection{Proof of Proposition \ref{prop diagram indices}} \label{sec proof diagram}

For simplicity, we will prove Proposition \ref{prop diagram indices} in the case where $E$ is the trivial line bundle. The general case can be proved analogously.

By definition of the maps  \eqref{eq phi phiE}, as in \cite{GHM}, the diagram
\[
\xymatrix{
C^*(X; L^2(E))^G_{\loc} \ar[r]^-{\oplus 0}& C^*(X)^G_{\loc}\\
 \ar[u]^-{\varphi_E} C^*(Z; L^2(E|_Z))^G  \ar[r]^-{\oplus 0}& C^*(Z)^G \ar[u]_{\varphi}
 }
\]
commutes. (This is in fact the only property of these maps that we use here.) For this reason, we disregard the top line in \eqref{eq diag indices}, and only work with Roe algebras on $Z$.

Let an element of $C^*(Z; L^2(Z))^G$ be given by a continuous kernel $\kappa\colon Z \times Z \to \C$ with finite propagation.

\begin{lemma} \label{lem one way}
For all  $\zeta \in L^2(Z) \otimes L^2(G)$, $g \in G$ and $z \in Z$,
\[
((\varphi_E(\kappa) \oplus 0)\zeta)(z,g) = \Bigl(
\int_G \tilTR(\kappa)(h) (h^{-1}g^{-1}\cdot \zeta(\relbar, gh)) \, dh \Bigr)(g^{-1}z).
\]
\end{lemma}
In this lemma, $ \zeta(\relbar, gh) \in L^2(Z)$, on which $G$ acts via its action on $Z$.
\begin{proof}
Consider the map \eqref{eq def j} in this setting,
%
\[
j\colon L^2(Z) \to L^2(Z) \otimes L^2(G).
\]
Then $\oplus 0$ is given by mapping operators on $L^2(Z)$ to the corresponding operators on $j(L^2(Z))$ by conjugation with $j$, and extending them by zero on the orthogonal complement of $j(L^2(Z))$.
Let $p\colon L^2(Z) \otimes L^2(G) \to j(L^2(Z))$ be the orthogonal projection. Then
\beq{eq phi E kappa 0}
\varphi_E(\kappa) \oplus 0 = j\circ \varphi_E(\kappa) \circ j^{-1} \circ p.
\eeq
One checks directly that for all $\zeta \in L^2(Z) \otimes L^2(G)$ and $z \in Z$,
\beq{eq p}
(j^{-1}\circ p)(\zeta) (z) = \int_G \chi(g^{-1}z) \zeta(z,g)\, dg.
\eeq
The lemma can now be proved via a straightforward computation involving \eqref{eq phi E kappa 0}, \eqref{eq p}, $G$-invariance of $\kappa$, and left invariance of the Haar measure on $G$.
%
%
\end{proof}

Next, fix $\kappa_G \in C^*_{\ker}(Z)^G$.
\begin{lemma} \label{lem other way}
For all $\zeta \in L^2(Z) \otimes L^2(G)$, $g \in G$ and $z \in Z$,
\begin{multline*}
((\varphi \circ a)(\kappa_G)\zeta)(z,g) = \\
\Bigl( \int_G W(\kappa_G)\bigl(\psi_1(z,g)^{-1}h \psi_1(z,g) \bigr) \zeta(\psi^{-1}(h\psi_1(z,g), \relbar))\, dh \Bigr)(\psi_2(z,g)).
\end{multline*}
\end{lemma}
\begin{proof}
This is a straightforward computation involving $G$-invariance of $\kappa_G$ and right invariance of the Haar measure on $G$.
\end{proof}

\begin{lemma} \label{lem Psi}
Let $\eta\colon X_1 \to X_2$ be a measurable bijection between measure spaces $(X_1, \mu_1)$ and $(X_2, \mu_2)$, such that $\eta^*\mu_2 = \mu_1$. Let $\sigma\colon X_1 \to G$ be any map. Define
\[
\Psi\colon C_c(G) \otimes \cK(L^2(X_2)) \to  C_c(G) \otimes \cK(L^2(X_1))
\]
by
\[
\bigl((\Psi(f)(g))u\bigr)(x) = \bigl(\bigl( \eta^* \circ f(\sigma(x)^{-1}g\sigma(x)) \circ (\eta^{-1})^* \bigr)u\bigr)(x)
\]
for all $f \in C_c(G) \otimes \cK(L^2(X_2))$, $g \in G$, $u \in L^2(X_1)$ and $x \in X_1$. Then the following diagram commutes:
\[
\xymatrix{
C_c(G) \otimes \cK(L^2(X_1)) \ar[d]_-{\tau_g \otimes 1} & C_c(G) \otimes \cK(L^2(X_2)) \ar[l]_-{\Psi}\ar[d]^-{\tau_g \otimes 1} \\
\cK(L^2(X_1)) \ar[r]^-{\eta^*} &
 \cK(L^2(X_2)).
}
\]
\end{lemma}
\begin{proof}
This is a straightforward computation, involving $G$-invariance of the measure $d(hZ_g)$ on $G/Z_g$.
\end{proof}
\begin{remark}
The map $\Psi$ in Lemma \ref{lem Psi} is not a homomorphism in general, unless $\sigma$ is constant.
\end{remark}

Applying Lemma \ref{lem Psi} with $X_1 = Z$, $X_2 = K\backslash G \times Y$ and $\sigma(z) = \psi_1(z,e)$, we obtain a commutative diagram
\beq{eq diag Psi}
\xymatrix{
C_c(G) \otimes \cK(L^2(Z)) \ar[d]_-{\tau_g \otimes 1} & C_c(G) \otimes \cK(H) \ar[l]_-{\Psi}\ar[d]^-{\tau_g \otimes 1} \\
\cK(L^2(Z)) \ar[r]^-{\eta^*} &
 \cK(H).
}
\eeq

\begin{proof}[Proof of Proposition \ref{prop diagram indices}]
As before, fix an element of $C^*(Z; L^2(Z))^G$ given by a continuous kernel $\kappa\colon Z \times Z \to \C$ with finite propagation, and $\kappa_G \in C^*_{\ker}(Z)^G$. Suppose that
\beq{eq kernels equal}
(\varphi \circ a)(\kappa_G) = (\varphi_E(\kappa) \oplus 0) \quad \in C^*(X; Z)^G.
\eeq
Then
Lemmas \ref{lem one way} and \ref{lem other way}, applied with $g=e$, imply that for all $\zeta \in L^2(Z) \otimes L^2(G)$  and $z \in Z$,
\begin{multline} \label{eq two ways}
\Bigl(
\int_G \tilTR(\kappa)(h) (h^{-1}\cdot \zeta(\relbar, h)) \, dh \Bigr)(z) \\
=
\Bigl( \int_G \eta^* \circ W(\kappa_G)\bigl(\psi_1(z,e)^{-1}h \psi_1(z,e) \bigr) \zeta(\psi^{-1}(h\psi_1(z,e), \relbar))\, dh \Bigr)(z).
\end{multline}
One has for all $z \in Z$ and $h \in G$,
\[
\psi(hz,h) = (h\psi_1(z,e), \eta(z)).
\]
(Recall that $\psi_1$ is the projection of $\psi$ onto $G$.)
Hence the right hand side of \eqref{eq two ways} equals
\[
\Bigl( \int_G \eta^* \circ W(\kappa_G)\bigl(\psi_1(z,e)^{-1}h \psi_1(z,e) \bigr) \circ (\eta^{-1})^*
 (h^{-1}\cdot \zeta(\relbar, h)) \, dh \Bigr)(z).
\]
Therefore, if $\zeta = u \otimes v$, for $u \in L^2(Z)$ and $v \in L^2(G)$, then
 \eqref{eq two ways} implies that for all  $z \in Z$,
 \begin{multline*}
\int_G v(h) \tilTR(\kappa)(h) (h^{-1}\cdot u) \, dh \Bigr)(z) \\
=  \Bigl( \int_G v(h) \bigl(\eta^* \circ W(\kappa_G)\bigl(\psi_1(z,e)^{-1}h \psi_1(z,e) \bigr) \circ (\eta^{-1})^*\bigr)
 (h^{-1}\cdot u) \, dh \Bigr)(z).
 \end{multline*}
 Hence for all $u \in L^2(Z)$, $h \in G$ and  $z \in Z$,
 \[
 \begin{split}
\bigl( \tilTR(\kappa)(h) u  \bigr)(z)
 &= \Bigl(\eta^* \circ W(\kappa_G)\bigl(\psi_1(z,e)^{-1}h \psi_1(z,e)\bigr) \circ (\eta^{-1})^*(u) \Bigr)(z) \\
 &= \bigl(\Psi(W(\kappa_G))(h)u \bigr)(z).
\end{split}
\]
So $\Psi(W(\kappa_G)) = \tilTR(\kappa)$, and commutativity of diagram \eqref{eq diag Psi} implies the claim.
\end{proof}

\subsection{Proof of Proposition \ref{prop coarse index}}


The isomorphism $C^*(X)^G_{\loc} \cong C^*_rG \otimes \cK$ used in \cite{GHM} to identify localised coarse indices with classes in $K_*(C^*_rG)$ is the map
\[
W \circ a^{-1} \circ \varphi^{-1},
\]
defined on a  dense subalgebra and extended continuously. Hence we explicitly have
\beq{eq expl def coarse ind}
\ind_G(\hat D) = (W_* \circ a_*^{-1} \circ \varphi_*^{-1})(\ind_G^{L^2(E)}(\hat D) \oplus 0) \quad \in K_0(C^*_rG).
\eeq
Therefore,
Lemma \ref{lem coarse idempotent} and
Proposition \ref{prop diagram indices} (see Remark \ref{rem L2Z L2E}) imply that
\[
\tau_g(\ind_G(\hat D)) 
= \tau_g \bigl(\tilTR \circ (\varphi_E)_*^{-1}([e] - [p_2]) \bigr).
\]
The trace map on the sub-algebra of trace-class operators in $\cK(L^2(E|_Z))$ induces the isomorphism
\[
K_*(C^*_rG \otimes\cK(L^2(E|_Z))) \cong K_*(C^*_rG ).
\]
Hence  Proposition \ref{prop coarse index} follows by Lemma \ref{lem taug TR Trg}.

\subsection{Proof of Theorem \ref{thm taug indG}}

\begin{proposition} \label{prop Sj2}
If the operators $e^{-t\tilde D^2}$ and $e^{-t\tilde D} \tilde D$ and $S_0^2$ and $S_1^2$  are $g$-trace class,
then  
\beq{eq Trg Sj2 Sj}
\Tr_g(S_0^2)-\Tr_g(S_1^2)=\Tr_g(S_0)-\Tr_g(S_1).
\eeq
\end{proposition}
\begin{proof}
We have  $S_0R=RS_1$, and hence
\[
\begin{split}
S_0-S_0^2&=S_0(1-S_0)= RS_1 \hat D_+;\\
S_1 - S_1^2 &= S_1(1-S_1)=  S_1\hat D_+ R.
\end{split}
\]
Because $e^{-t\tilde D^2}$ and $e^{-t\tilde D} \tilde D$ are $g$-trace class, Lemma 5.3 and Proposition 5.7 in \cite{HWW} imply that $S_0$ and $S_1$ are $g$-trace class.  So the operators $RS_1 \hat D_+$ and $S_1 \hat D_+ R$ are $g$-trace class.

By Lemma \ref{lem comp Sj},
\beq{eq S1D}
S_1\hat D_+ =
\varphi_1 \tilde S_{1}\tilde D_+ \psi_1 - \varphi_1 \tilde S_1 \sigma \psi_1' - \varphi_1' \sigma \tilde Q \tilde D_+ \psi_1 + \varphi_1' \sigma \tilde Q  \sigma \psi_1'    - \varphi_2' \sigma Q_CD_{C, +} \psi_2 + \varphi_2' Q_C \sigma\psi_2'.
\eeq
Since $\tilde S_1$ and $\sigma^{-1}\tilde S_1 \tilde D_+ $ are $g$-trace class by assumption,  $\varphi_j'$ has disjoint support from  $\psi_j$, and all operators occurring are pseudo-differential operators, and therefore have smooth kernels off the diagonal, we find that
$\sigma^{-1} S_1\hat D_+$ is $g$-trace class. (And the last four terms on the right hand side of \eqref{eq S1D} have $g$-trace zero.)
And $R\sigma$ is  has a distributional kernel, so
 Lemma \ref{lem ST TS} implies that
\[
\Tr_g(RS_1\hat D_+)=\Tr_g(R\sigma \sigma^{-1}S_1\hat D_+)=\Tr_g(\sigma^{-1} S_1\hat D_+ R \sigma) = \Tr_g(S_1\hat D_+ R).
\]
Hence
\eqref{eq Trg Sj2 Sj} follows.
\end{proof}

Theorem \ref{thm taug indG} follows from Propositions \ref{prop Sj2 g trace cl}, \ref{prop coarse index} and \ref{prop Sj2}.

\section{Non-invertible $D_N$} \label{sec DN not inv}

We have so far assumed that the Dirac operator $D_N$ on the boundary $N$ is invertible. We now discuss how that assumption can be weakened to the assumption that $0$ is isolated in the spectrum of $D_N$. The arguments are related to those in Section 6 of \cite{HWW}.

\subsection{A shifted Dirac operator}
Let $\varepsilon > 0$ be such that $([-2\varepsilon, 2\varepsilon] \cap \spec (D_N))\setminus \{0\} = \emptyset$. Let $\psi \in C^{\infty}(\hat M)^G$ be a nonnegative function such that
\[
\begin{split}
\psi(n,u) &= \left\{
\begin{array}{ll}
u & \text{if $n \in N$ and $u \in (1/2, \infty)$;} \\
0 & \text{if $n \in N$ and $u \in (0,1/4)$;}
\end{array}
\right. \\
\psi(m) &= 0 \quad \text{if $m \in M \setminus U$}.
\end{split}
\]
(Recall that $U \cong N \times (0,1]$ is a neighbourhood of $N$ in $M$.)

As in Section 6 of \cite{HWW}, we
consider the $G$-equivariant, odd, elliptic operator
\[
\hat D_{\varepsilon} := e^{\varepsilon \psi} \hat D e^{-\varepsilon \psi}.
\]
The operator $\hat D_{\varepsilon}$ is $G$-equivariant, essentially self-adjoint, odd-graded and elliptic.
Its restriction to $\hat M \setminus M$ equals
\beq{eq Deps}
\sigma \Bigl(-\frac{\partial}{\partial u} + D_N + \varepsilon \Bigr).
\eeq
It therefore satisfies the condition \eqref{eq D2 pos},
and  has a well-defined index
\[
\ind_G(\hat D_{\varepsilon}) \in K_0(C^*_rG).
\]
Let $a_1$ be as in \eqref{eq SM}.
Theorem \ref{thm taug indG} generalises as follows.
\begin{theorem}\label{thm taug indG non inv}
Suppose that  $\hat D_{\varepsilon}$ is $g$-Fredholm, and that the heat kernel decay \eqref{eq heat kernel decay} holds for the operators mentioned.
If either
\begin{itemize}
\item[(a)] 
$G/Z_g$ is compact; or
\item[(b)] 
$G = \Gamma$ is discrete and finitely generated, and \eqref{eq growth Gamma} holds for a $k<\frac{2 a_1 \varepsilon}{3}$,
\end{itemize}
then 
\[
\tau_g(\ind_G(\hat D_{\varepsilon})) = \ind_g(\hat D_{\varepsilon}).
\]
\end{theorem}
Conditions for $\hat D_{\varepsilon}$ to be $g$-Fredholm were given in Theorem 6.1 and Corollary 6.2 in \cite{HWW}.
 
Corollary \ref{cor index taug} also generalises to this setting. This involves Corollary 6.2 in \cite{HWW}.

\subsection{A shifted parametrix}

Let $\tilde \psi$ be any smooth, $G$-invariant extension of $\psi|_M$ to the double $\tilde M$ of $M$.  
As in Subsection 6.3 of \cite{HWW}, we use the operators
\[
\begin{split}
\tilde D_{\varepsilon} &= e^{\varepsilon \tilde \psi} \tilde D e^{-\varepsilon \tilde \psi};\\
\tilde Q_{\varepsilon} &:= \frac{1-e^{-t\tilde D_{\varepsilon, -}\tilde D_{\varepsilon, +} }}{\tilde D_{\varepsilon, -}\tilde D_{\varepsilon, +}}\tilde D_{\varepsilon, -};\\
\tilde S_{\varepsilon, 0}&:= 1-\tilde Q_{\varepsilon} \tilde D_{\varepsilon, +} = e^{-t\tilde D_{\varepsilon, -}\tilde D_{\varepsilon, +} };\\
\tilde S_{\varepsilon, 1} &:= 1- \tilde D_{\varepsilon, +}  \tilde Q_{\varepsilon}=  e^{-t\tilde D_{\varepsilon, +}\tilde D_{\varepsilon, -}}.
\end{split}
\]
Let $D_{C, \varepsilon}$ be the restriction of $\hat D_{\varepsilon}$ to $N \times (1/2, \infty)$, and
let $Q_{C, \varepsilon}$ be the inverse of its self-adjoint closure, restricted to sections of $E_-$.

Let the functions $\varphi_j$ and $\psi_j$ be as in Subsection \ref{sec Sj}, with the difference that they change values between $0$ and $1$ on the interval $(1/2, 1)$ rather than on $(0,1)$.
 Set
\[
\begin{split}
R_{\varepsilon} &:= \varphi_1\tilde Q_{\varepsilon}  \psi_1 + \varphi_2 Q_{C, \varepsilon} \psi_2;\\
S_{\varepsilon, 0} &:= 1-R_{\varepsilon}\hat D_{\varepsilon, +};\\
S_{\varepsilon, 1} &:= 1-\hat D_{\varepsilon, +}R_{\varepsilon}.
\end{split}
\]

From this point on, the proof of Theorem \ref{thm taug indG non inv} is analogous to the proof of Theorem \ref{thm taug indG}. The starting point is that, 
as in Lemma \ref{lem comp Sj}, 
\[
\begin{split}
S_{\varepsilon, 0} &= \varphi_1 \tilde S_{\varepsilon, 0} \psi_1 + \varphi_1 \tilde Q_{\varepsilon}\sigma  \psi_1' + \varphi_2 Q_{C, \varepsilon} \sigma \psi_2';\\
S_{\varepsilon, 1} &= \varphi_1 \tilde S_{\varepsilon, 1} \psi_1 - \varphi_1'\sigma  \tilde Q_{\varepsilon}  \psi_1 - \varphi_2' \sigma Q_{C, \varepsilon} \psi_2.\\
\end{split}
\]
As noted in Subsection 6.3 of \cite{HWW}, the arguments showing that 
 $S_0$ and $S_1$ are $g$-trace class immediately generalise to show 
 that $S_{\varepsilon, 0}$ and $S_{\varepsilon, 1}$  are $g$-trace class. 
Similarly, Propositions \ref{prop Sj2 g trace cl}, \ref{prop coarse index} and \ref{prop Sj2} generalise to the current situation, and imply Theorem \ref{thm taug indG non inv}.

\bibliographystyle{plain}

\bibliography{mybib}

\end{document}